\documentclass{amsart}

\usepackage[margin=25mm]{geometry}
\usepackage[ngerman,english]{babel}
\usepackage{pgf}
\usepackage{tikz}
\usetikzlibrary{shapes,shapes.geometric,arrows,fit,calc,positioning,arrows,automata,}
\usepackage[utf8]{inputenc} 
\usepackage[T1]{fontenc} 

\usepackage{amsthm}
\usepackage{amsmath}
\usepackage{amssymb, latexsym}

\usepackage{thmtools}
\usepackage{thm-restate}

\usepackage{array} 
\usepackage{tabularx}
\usepackage{graphicx} 

\usepackage{enumerate}
\usepackage{hyperref}
\usepackage{url}

\allowdisplaybreaks

\newcommand{\N}{\mathbb{N}}
\newcommand{\Z}{\mathbb{Z}}
\newcommand{\R}{\mathbb{R}}
\newcommand{\Q}{\mathbb{Q}}
\newcommand{\C}{\mathbb{C}}
\renewcommand{\P}{\mathbb{P}}
\newcommand{\U}{\mathbb{U}}

\newcommand{\bfw}{\mathbf{w}}
\newcommand{\bfv}{\mathbf{v}}

\DeclareMathOperator{\e}{\mathrm{e}}

\DeclareMathOperator{\ind}{\mathbf{1}}

\DeclareMathOperator{\tr}{\mathrm{tr}}
\DeclareMathOperator{\card}{\mathrm{card}}

\newcommand{\floor}[1]{\left\lfloor #1 \right\rfloor}
\newcommand{\ceil}[1]{\left\lceil #1 \right\rceil}
\newcommand{\abs}[1]{\left| #1 \right|}
\newcommand{\norm}[1]{\left\| #1 \right\|}

\newcommand{\rb}[1]{\left( #1 \right)}

\newcommand{\inda}{\ind_{[a(n_{\ell}) = \alpha]}}
\newcommand{\indb}{\ind_{[b(n_{\ell}) = \alpha]}}

\newcommand{\indbi}{\ind_{[b_i(n_{\ell}) = \alpha]}}
\newcommand{\cA}{\mathcal{A}}

\newtheorem{theorem}{Theorem}[section]
\newtheorem{lemma}[theorem]{Lemma}
\newtheorem{proposition}[theorem]{Proposition}
\newtheorem{corollary}[theorem]{Corollary}

\theoremstyle{definition}
\newtheorem{definition}[theorem]{Definition}
\newtheorem{example}[theorem]{Example}

\newtheorem{conjecture}[theorem]{Conjecture}
\newtheorem{fact}[theorem]{Fact}

\theoremstyle{remark}
\newtheorem{remark}[theorem]{Remark}

\numberwithin{equation}{section}

\begin{document}

\title{(Logarithmic) Densities for automatic sequences along primes and squares}


\author{Boris Adamczewski}
\address{Univ Lyon, Universit\'e Claude Bernard Lyon 1, CNRS UMR 5208, Institut Camille Jordan, 43 blvd. du 11 novembre 1918, F-69622 Villeurbanne cedex, France}
\curraddr{}
\email{boris.adamczewski@math.cnrs.fr}
\thanks{The first and third author were supported by the European Research Council (ERC) under the European Union's Horizon 2020 research and innovation program under the Grant Agreement No 648132.}

\author{Michael Drmota}
\address{Institut f\"ur Diskrete Mathematik und Geometrie TU Wien, Wiedner Hauptstr. 8–10, 1040
Wien, Austria}
\curraddr{}
\email{michael.drmota@tuwien.ac.at}
\thanks{The second and third author were supported by the Fond zur F\"orderung der wissenschaftlichen Forschung (FWF), grant SFB F55-02
"Subsequences of Automatic Sequences and Uniform Distribution"}

\author{Clemens M\"ullner}
\address{Institut f\"ur Diskrete Mathematik und Geometrie TU Wien, Wiedner Hauptstr. 8–10, 1040
Wien, Austria}
\curraddr{}
\email{clemens.muellner@tuwien.ac.at}



\date{}

\dedicatory{}

 \begin{abstract}
  In this paper we develop a method to transfer density results for primitive automatic sequences to logarithmic-density results for general automatic sequences. As an application we show that the logarithmic densities of any automatic sequence along squares $(n^2)_{n\geq 0}$ and primes $(p_n)_{n\geq 1}$ exist and are computable. Furthermore, we give for these subsequences a criterion to decide whether the densities exist, in which case they are also computable.
  In particular in the prime case these densities are all rational.  We also deduce from a recent result of the third author and Lema\'nczyk that all subshifts  generated by  automatic sequences  are orthogonal to any bounded multiplicative aperiodic function. 

  \end{abstract}

\maketitle

\section{Introduction}
  
Automatic sequences are sequences $a(n)$ on a finite alphabet that
are the output of a finite automaton (where the input is the sequence of digits of $n$ in some base $k\ge 2$).
Equivalently, they can also be defined as codings of fixed points of morphisms of constant length.
These kind of sequences have received a lot of attention during the last 15 or 20 years (see for example
the book by Allouche and Shallit \cite{Allouche2003}).
In particular there are very close relations to number theory, dynamical systems, and algebra. 
The most prominent examples of automatic
sequences are the Thue-Morse sequence $t(n)$ and the Rudin-Shapiro sequence $r(n)$.\footnote
{The Thue-Morse sequence can be defined by $t(n) = s_2(n) \bmod 2$, where $s_2(n)$ denotes the
number of $1$'s in the binary expansion of $n$, and the Rudin-Shapiro sequence by 
$r(n) = B_{11}(n) \bmod 2$, where $B_{11}(n)$ denotes the number of consecutive $11$-blocks in the
binary expansion of $n$.}

Automatic sequences are {\it deterministic sequences} in the sense that they generate a topological dynamical system (subshift) with \emph{zero entropy}. Stated differently, their subword complexity, that is, 
the number of different subwords of length $l$, is subexponential. Actually the subword complexity 
of automatic sequences is at most
linear in $l$, which is the lowest possible growth order if we exclude just eventually periodic sequences, which have bounded subword
complexity. Deterministic sequences have been intensively studied within the last few years in relation to 
the Sarnak conjecture \cite{Sarnak2011}, which says that deterministic sequences $d(n)$ are asymptotically orthogonal to the M\"obius function
$\mu(n)$:
\[
\sum_{n\le x} d(n) \mu(n) = o(x) \qquad (x\to\infty).
\]
This conjecture, which is related to the Chowla conjecture (see for example~\cite{Tao2017} or ~\cite{Gomilko2018}), is in general open.
There is, however, big progress on the logarithmic version of the Chowla conjecture by Tao~\cite{Tao2016a} and Tao and Ter\"av\"ainen~\cite{Tao2018}, and also for the logarithmic version of the Sarnak conjecture by Frantzikinakis and Host~\cite{Frantzikinakis2018}.
For a relatively recent survey on the Sarnak conjecture see~\cite{Ferenczi2018}.
Recently the last author verified the Sarnak conjecture for all automatic sequences $d(n) = a(n)$ \cite{Muellner2017},
and even more recently he extended together with Lemanczyk~\cite{Lemanczyk2018} the orthogonality relation to multiplicative functions $f(n)$ with
$|f(n)|\le 1$ (and some mild but natural conditions on $a(n)$): 
\[
\sum_{n\le x} a(n) f(n) = o(x) \qquad (x\to\infty).
\]

The Sarnak conjecture (and the above variant for multiplicative functions) is closely related to 
the asymptotic properties of 
\[
\sum_{n\le x} d(n) \Lambda(n),
\]
where $\Lambda(n)$ denotes the von Mangoldt $\Lambda$-function\footnote{$\Lambda(n) = \log p$ for
prime powers $n = p^k$ and $\Lambda(n) = 0$ else.} and one usually conjectures that these
sums are asymptotically of the form $c x (1 +o(1))$ for some constant $c$. This property is very
close to {\it prime number theorems} of the kind
\[
\{ p \le x : p\in \mathbb{P}, \ d(p) = \alpha \} = c \,\frac x{\log x} (1+o(1)).
\]
For the Thue-Morse sequence $t(n)$ such a prime number theorem 
\[
\{ p \le x : p\in \mathbb{P}, \ t(p) = 0 \} \sim \{ p \le x : p\in \mathbb{P}, \ t(p) = 1 \} \sim  \frac 12 \frac x{\log x}
\]
was already conjectured in 1968/69 by Gelfond \cite{Gelfond1967} (in a slightly more general form). This conjecture was
finally proved in 2009 by Mauduit and Rivat \cite{Mauduit2010} in a breakthrough paper. 
Gelfond conjectured, too, that the Thue-Morse sequence behaves nicely along polynomial subsequences:
\[
\{ n \le x : t(P(n)) = 0 \}  \sim \{ n \le x : t(P(n)) = 1 \} \sim  \frac x2 ,
\]
where $P(x)$ is a non-negative integer valued polynomial. This question turned out to be even more challenging.
It was finally solved (again) by Mauduit and Rivat \cite{Mauduit2009} for quadratic polynomials but for polynomials of
degree at least 3 there is only partial information available \cite{Drmota2011}; the question by Gelfond is still open.

We cannot expect such {\it strong} results for general automatic sequences. For example, if $a(n)$ denotes the
leading digit in the $k$-ary expansion of $n$ (with $k\ge 3$) then neither the {\it densities}
\[
d(a(n),\alpha) = \lim_{x\to\infty} \frac{ \{n\le x : a(n) = \alpha \} }x =   \lim_{x\to\infty} \frac 1x \sum_{n\le x} {\bf 1}_{[a(n) = \alpha]}  \,,
\]
nor the densities along primes 
\[
d(a(p_n),\alpha) = \lim_{x\to\infty} \frac{ \{n\le x : a(p_n) = \alpha \} }x =   \lim_{x\to\infty} \frac 1x \sum_{n\le x} {\bf 1}_{[a(p_n) = \alpha]}  
\]
exist for $1\leq \alpha <k$. Nevertheless --  and this a general property for automatic sequences (see \cite{Cobham1972}) -- the  
{\it logarithmic densities}
\[
d_{\log }(a(n),\alpha) = \lim_{x\to\infty} \frac 1{\log x}  \sum_{n\le x}  \frac 1n {\bf 1}_{[a(n) = \alpha]}=\frac{\log (1+1/\alpha)}{\log k}
\]
exist. The question whether a density exists or only a logarithmic density exists depends mainly on the behavior of the final strongly connected components of the corresponding finite automata.
Furthermore, if densities exist they can be explicitly computed and are rational numbers \cite{Cobham1972}.
In the case, when only logarithmic densities exist, this is not that clear. 

The above mentioned results are related to density results of special automatic sequences along 
special subsequences: the subsequence of primes $p$ and the subsequence of squares $n^2$. 
The purpose of the present paper is to study quite general subsequences of automatic sequences and to give
answers to the question whether a density or logarithmic density along subsequences exist.
In particular we will give a complete answer for the subsequence of primes and squares (Theorems~\ref{Th2} and \ref{th_squares_log}).
For these cases we will show that logarithmic densities always exist. In other terms this means that, for every
automatic sequence $a(n)$, the
following limits always exist:
\[
\lim_{x\to \infty} \frac 1{\log x} \sum_{n\le x} \frac 1n a(n) \Lambda(n) \quad \mbox{and}\quad
\lim_{x\to \infty} \frac 1{\log x} \sum_{n\le x} \frac 1n a(n^2).
\]
and we can decide, when the non-logarithmic versions hold. 
We want to add that for some special classes of automatic sequences, that is, invertible automatic sequences or
automatic sequences related to block-additive functions, this is already known \cite{Drmota2012}, \cite{Muellner2018}.

\medskip

In order to state our main results we have to introduce some notation. First of all we will only consider
strictly increasing subsequences $(n_\ell)_{\ell\ge 0}$ of the positive integers that behave as
\begin{align}\label{eq_gamma}
n_\ell = \ell^\gamma L(\ell),
\end{align}
where $\gamma \ge 1$ and $L(n)$ is slowly varying in the sense that
\begin{align}\label{eq_lim_L}
\lim_{\ell\to\infty} \frac{L(\ceil{\delta \ell})}{L(\ell)} = 1
\end{align}
for all $0< \delta < 1$. 
Such sequences $(n_{\ell})_{\ell\ge 0}$ are called \emph{regularly varying sequences}, see Section~\ref{sec_n_l}.
The sequence of primes, polynomial sequences, and Piatetski-Shapiro sequences (\emph{i.e.,} 
$\lfloor n^c \rfloor$, where $c>1$) 
provide relevant examples of regularly varying sequences.

As mentioned above every automatic sequence $a(n)$ can be generated by a finite automaton. Without loss
of generality we can assume that this automaton is minimal (see \cite{Allouche2003}). This automaton can be seen as
a directed graph, possibly with loops and multiple edges, where every vertex (or state) has out-degree $k$
and for every vertex the out-going edges are labeled by $0,1,\ldots,k-1$.\footnote
{We will be more precise in Section~\ref{sec:automatic}} The set $\{ 0,1,\ldots,k-1\}$ is the \emph{input alphabet}. 
One vertex of this graph is distinguished
as the initial state. Clearly, this graph decomposes into strongly connected components.\footnote{A strongly connected component is a maximal subset of vertices, such that for any two vertices $u, v$ in this set, there exists a directed path from $u$ to $v$ and vice-versa.}
A strongly connected component is called final if there is no edge from this component to another one.
We will say that an automatic sequence is primitive and prolongable if the 
directed graph of the corresponding minimal automaton is strongly connected and the initial state has a $0$-labeled loop.

Our first result says that it is sufficient to consider such automatic sequences.

\begin{theorem}\label{Th1}
Suppose that $(n_\ell)_{\ell \ge 0}$ is a regularly varying sequence and suppose that for any primitive 
and prolongable automatic sequence $\tilde a(n)$ the 
densities along the subsequence $(n_\ell)$ 
\[
d(\tilde a(n_\ell),\alpha) := \lim_{x\to\infty}  \frac{\{\ell \le x : \tilde a(n_\ell) = \alpha \} }x
\] 
 exist. 

Then the two following properties hold. 

\begin{itemize}
\item[{\rm (i)}] Then for every automatic sequence $a(n)$ the logarithmic densities
\[
d_{\log }(a(n_\ell),\alpha) := \lim_{x\to\infty}\frac{1}{\log x}  \sum_{\ell \leq x} \frac{1}{\ell} {\bf 1}_{[a(n_\ell) = \alpha]}  
\]
exist and can be explicitly computed.

\item[{\rm (ii)}] Furthermore, if the densities along the subsequence $n_\ell$ corresponding 
to the automatic sequences that are generated by the final strongly connected components 
of the directed graph are all equal then the densities
\[
d(a(n_\ell),\alpha):= \lim_{x\to\infty}  \frac{\{\ell \le x :  a(n_\ell) = \alpha \} }x
\]
exist and are equal to the corresponding densities of the final strongly connected components. 
\end{itemize}
\end{theorem}

This theorem will be now applied to primes and squares.
We start with primes and note that it was already shown in \cite{Muellner2017} that primitive and prolongable
automatic sequences along the primes have densities that are all computable rational numbers.
Together with Theorem~\ref{Th1} this solves the problem for primes completely.

\begin{theorem}\label{Th2}
For every automatic sequence $a(n)$ the logarithmic densities $d_{\log }(a(p_n),\alpha)$ of the subsequence 
along prime numbers exist and are computable. 
Furthermore, if the densities along primes on those automatic sequences that correspond to the
final strongly connected components coincide
then the densities $d(a(p_n),\alpha)$ exist and are computable rational numbers.
\end{theorem}

The same result hold for subsequences along squares. However, in this case we have to check 
the assumption on primitive and prolongable automatic sequences, see Section~\ref{sec:squares}. 
In both cases of primes and squares we are able to compute the densities for primitive and prolongable automatic sequence.
As an example, we compute the densities of the paper-folding sequence along primes and squares in Section~\ref{sec:computability}. 
The densities of $0$ and $1$ in the paper-folding sequence along primes are both $1/2$, whereas the density of $1$ in 
the paper-folding sequence along squares is $1$.

\begin{theorem}\label{th_squares_log}
For every automatic sequence $a(n)$ the logarithmic densities $d_{\log }(a(n^2),\alpha)$ of the 
subsequence along squares exist and are computable. 
Furthermore, if the densities along squares on those automatic sequences that correspond to the 
final strongly connected components coincide then the densities $d(a(n^2),\alpha)$ exist and are also computable. 
If the input base $k$ is prime, then these densities are rational numbers.
\end{theorem}

As a simple application, we can compute the logarithmic densities of the leading digit of primes and polynomials $P(n)$ with integer coefficients,
\begin{align*}
	d_{\log }(a(p_n),\alpha) = d_{\log}(a(P(n)), \alpha) = \frac{\log (1+1/\alpha)}{\log k}.
\end{align*}

Theorems~\ref{Th2} and \ref{th_squares_log} suggest that the subsequences of primes and squares are
{\it similar}, at least for the question of the existence of (logarithmic) densities of automatic sequences along these
subsequences. As we will see in Section~\ref{sec:primes} and Section~\ref{sec:squares}, they share several distribution properties. However, it seems that
there are still {\it fundamental differences}. For example in the prime case there is the following quite 
unexpected property.
	\begin{theorem}\label{th_primes_mod_m}
		For any automatic sequence $a(n)$ there exists a computable positive integer $m$ such that, for all $\alpha$, 
		$d_{\log}(a(p_n), \alpha)$ is equal to the logarithmic density of $a(n)$ along the integers $n$  
satisfying $(n, m) = 1$.
	\end{theorem}
\begin{remark}
This theorem also works for densities in the sense that if the density exists for one of them, then it also exists for the other one and they coincide.
\end{remark}
\begin{remark}
This theorem applies for example to the residue of any block-additive function $f \bmod m$ satisfying $(k-1, m) = 1$ and $(\gcd(f(n)_{n\in \N}), m) = 1$, as this sequence distributes uniformly along any arithmetic progression, which follows from~\cite[Proposition 3.15]{Muellner2018}. However, this result was already put as a remark in~\cite{Muellner2017}, without a proof. 
\end{remark}

We could not find a corresponding property for squares. We expect that the {\it deeper reason} for this difference is that
primes have a quasi-random behavior that is not present for squares. We leave it as an open problem to clarify this phenomenon.

We would like to thank the anonymous referee, who suggested to also consider the upper and lower densities of subsequences of automatic sequences similarly to~\cite{Bell2020}.
Indeed, very similar methods can be used to treat slowly varying subsequences of automatic sequences.
We sketch a proof of the corresponding results, that is, Theorem~\ref{th_upper_density} and Corollary~\ref{co_upper_primes}, in Appendix~\ref{sec_upper}.

To end this introduction, let us mention two possible directions for further research.  
First, we recall an analogous problem for Piatetski-Shapiro sequences $\floor{n^c}$, with 
$1 < c < 7/5$, which has already been solved in  \cite{Deshouillers2012}.  Indeed, these authors proved that, 
for every automatic sequence $a(n)$ and for every $c\in (1,7/5)$, the logarithmic densities 
$d_{\log}(a(\lfloor n^c\rfloor), \alpha)$ exist and are equal to the logarithmic densities of $a(n)$. 
Furthermore, the densities $d(a(\floor{n^c}),\alpha)$
exist if and only if the densities $d(a(n),\alpha)$ exist, in which case they are equal. 
We conjecture that such a result should also hold for all $1<c < 2$.

\begin{conjecture}
For every automatic sequence $a(n)$ and for every $c\in (1,2)$, the logarithmic densities 
$d_{\log}(a(\lfloor n^c\rfloor), \alpha)$ exist and are equal to the logarithmic densities $d_{\log}(a(n), \alpha)$. 
Furthermore, the densities $d(a(\floor{n^c}),\alpha)$
exist if and only if the densities $d(a(n),\alpha)$ exist, in which case they are equal. 
\end{conjecture}

A second possible direction of research is related to \emph{morphic sequences}.  
As we have already mentioned, automatic sequences are codings of fixed points of morphisms of constant length.
Thus, they are a special case of morphic sequences, which can be obtained as codings of fixed points of general morphisms.
It is already known that the logarithmic densites of morphic sequences exist~\cite{Bell2008} and it seems reasonable to assume that Theorem~\ref{Th1} could be extended to 
morphic sequences. However, it seems much more challenging to prove the analogs of Theorems~\ref{Th2} and \ref{th_squares_log} in this wider framework.

\subsection{Plan of the paper}  
We start with a short section on regularly varying functions (Section~\ref{sec_n_l}) and proceed
with a longer background section on properties of automatic sequences (Section~\ref{sec:automatic}).
In particular we discuss (partly new) structural results that will be needed for the proof of
Theorem~\ref{Th1} that will be given in Section~\ref{sec:transfer}.
In Section~\ref{sec_primitive} we present a strategy how one can check that densities for primitive and prolongable automatic sequences
exist so that Theorem~\ref{Th1} can be applied.
Section~\ref{sec:primes} is then devoted to the case of prime numbers (Theorem~\ref{Th2}) and Section~\ref{sec:squares} to
the case of squares (Theorem~\ref{th_squares_log}). 
Finally, Section~\ref{sec:computability} is devoted to the problem, how densities along primes and squares 
can be actually computed (including some examples). 
In Appendix \ref{sec:dynamical}, we gather some implications to dynamical systems. In particular, we deduce  
from a recent result of the third author and Lema\'nczyk that all subshifts  generated by  automatic sequences  are orthogonal to any bounded multiplicative aperiodic function (Corollary \ref{cor:ortho}). 
Finally, Appendix~\ref{sec_upper} is devoted to a sketch for how one can treat upper and lower densities of automatic sequences along slowly varying subsequences.

  \subsection{Notation}
  In this paper we let $\N$ denote the set of positive integers and we use the abbreviation 
  $\e(x) = \exp(2\pi ix)$ for any real
number $x$.

For two functions, $f: \R \to \R$ and $g: \R \to \R_{>0}$ such that $f /g$ is
bounded, we write $f = O(g)$ or $f \ll g$.
If in fact $\abs{f(x)} \leq g(x)$ for all $x$, we write $f = O^{*}(g)$.
Furthermore, we write $f = o(g)$ if $\lim_{x \to \infty} f(x)/g(x) = 0$. 
We also write $f \sim g$ if $\lim_{x \to \infty} f(x) / g(x) = 1$.

We let $\floor{x}$ denote the floor function and $\ceil{x}$ denote the ceiling function.

Moreover we let $\varphi(n)$ denote the Euler totient function.
Finally, we let $\P$ denote the set of prime numbers and by $\pi(x)$ the number of prime numbers smaller than or equal to $x$.
  
  \section{Regularly varying functions}\label{sec_n_l}
  We discuss in this section some properties of subsequences $(n_{\ell})_{\ell \in \N}$ satisfying \eqref{eq_gamma} and \eqref{eq_lim_L} for some $\gamma \geq 1$ and $L: \N \to \R$.
  We define a new function $f: \R_{\geq 1} \to \R_{>0}, f(x) = n_{\ceil{x}}$.
  It follows directly that $f$ is measurable. Furthermore, we have for any $\delta>0$ that
  \begin{align*}
  	\lim_{x \to \infty} \frac{f(\delta x)}{f(x)} = \delta^{\gamma},
  \end{align*}
  i.e. it is regularly varying of index $\gamma$ (see \cite{Bingham1987} for background on regularly varying functions).
  If $\gamma = 1$, $f$ is called slowly varying.
  
  We consider the generalized inverse function of $f$, $g(x) := \inf \{y \in [1,\infty): f(y) > x\}$.
  In particular, we have $g(N) = \#\{\ell \in \N: n_{\ell} \leq N\}$.
  One has by \cite[Theorem 1.5.12]{Bingham1987} that $g$ is regularly varying of index $1/\gamma$, i.e. for every $\delta > 0$
  \begin{align*}
  	\lim_{N \to \infty}\frac{g(\delta N)}{g(N)} = \delta^{\beta},
  \end{align*}
  where we set $\beta := 1/\gamma$.
  	\begin{lemma}\label{le_log_g}
  		With the notation from above, we have
  		\begin{align*}
  			\log(g(N)) \sim \beta \log N.
  		\end{align*}
  	\end{lemma}
  	\begin{proof}
  		As $g$ is regularly varying of index $\beta$ we can write it as
  		\begin{align*}
  			g(x) = x^{\beta} \cdot c(x) \cdot \exp\rb{\int_{1}^{x} \frac{\varepsilon(u)}{u} du},
  		\end{align*}
  		where $c(x)$ converges to some $c \in (0,\infty)$ and $\varepsilon(x)$ converges to $0$ for $x \to \infty$ (see \cite[Theorem 1.3.1]{Bingham1987}).
  		In particular we have that
  		\begin{align*}
  			\int_1^{x} \frac{\varepsilon(u)}{u} du = o(\log x),
  		\end{align*}
  		which finishes the proof by basic properties of the logarithm. 
  	\end{proof}

  \section{Automatic sequences}\label{sec:automatic}
  
Let us now describe the precise setting of our study.  
First we give some definitions related to automata which can also be found in \cite{Allouche2003}.

A sequence $(a_n)_{n\geq 0}$ with values in a finite set 
is \emph{$k$-automatic} if it can be generated by a finite automaton.  
This means that there exists a finite-state machine 
(a deterministic finite automaton with output) that takes as input the base-$k$ 
expansion of $n$ and produces as output the symbol $a_n$.  
We use the following convention.  
Inputs are read from left to right, that is,
starting from the most significant digit. 

\subsection{Formal definition of \texorpdfstring{$k$}{k}-automatic sequences} 
Throughout this paper, we will use the following notation.  
An alphabet $A$ is a finite set of symbols, also called letters. A finite word over $A$ is a 
finite sequence of letters in $A$  or, equivalently, an element of $A^*$, the free monoid  
generated by $A$.  
The length of a finite word $\bfw$, that is, the number 
of symbols in $\bfw$, is denoted by $\vert \bfw\vert$. 
We let $\epsilon$ denote the empty word, the neutral element of $A^*$.
Let $k\ge 2$ be a natural number.  
We let $\Sigma_k$ denote the alphabet $\left\{0,1,\ldots,k-1\right\}$.
Given a positive integer $n$, we set 
$(n)_k:=w_r w_{r-1}\cdots w_0$ for the canonical
base-$k$ expansion of $n$ (written from most to least significant digit),
which means 
that $n=\sum_{i=0}^r w_i k^{i}$ with $w_i\in \Sigma_k$ and $w_r\not= 0$.
Note that by convention $( 0)_k:= \epsilon$. 
Conversely, if $\bfw:= w_0 \cdots w_r$  is a finite word over the alphabet $\Sigma_k$, 
we set $[\bfw]_k:=\sum_{i=0}^r w_{r-i} k^{i}$.  
Furthermore, we let $(n)_k^t$ denote the unique word $\bfw$ of length $t$ such that  $[\bfw]_k \equiv n \bmod k^{t}$.
\begin{example}
  We find $(37)_2 = 100101, (37)_2^4 = 0101$ and $[010110]_2 = 22$.
\end{example}

\begin{definition}
A $k$-deterministic finite automaton, or $k$-DFA for short, is a quadruple $\cA = (Q,\Sigma_k, \delta, q_0)$, 
where $Q$ is a finite set of states, $\Sigma_k:=\{0,1,\ldots,k-1\}$ is the finite input alphabet,
$\delta: Q \times \Sigma \rightarrow Q$ is the transition function and $q_0 \in Q$ is the initial state. 
A $k$-DFAO $\cA = (Q,\Sigma,\delta,q_0, \Delta, \tau)$ is a $k$-DFA endowed with an additional output function 
$\tau: Q \to \Delta$, where $\Delta$ is the alphabet of output symbols. 
\end{definition}

We extend $\delta$ to a function $\delta: Q\times\Sigma^{*}\to Q$ as follows. 
Given a state $q$ in $Q$ and a finite 
word $\bfw:=w_1 w_2 \cdots w_n$ over the alphabet $\Sigma_k$, 
we define $\delta(q,\bfw)$ recursively by 
$\delta(q,\epsilon)=q$
and $\delta(q,\bfw)=\delta(\delta(q,w_1w_2\cdots w_{n-1}),w_n)$.  
Hence computing $\delta(q,\bfw)$ involves $\abs{\bfw}$ ``steps'' for every $\bfw\in \Sigma^{*}$.

\begin{definition}\label{def:automatic_sequence}
  We say that a sequence $(a(n))_{n\geq 0}$ is a \emph{$k$-automatic} sequence if  there exists a $k$-DFAO 
  $\cA = (Q, \Sigma_k, \delta, q_0,\Delta, \tau)$ 
  such that $a_n = \tau(\delta(q_0,(n)_k))$. 
	If $\Delta = Q$ and $\tau = id$, then we call $(a(n))_{n\geq 0}$, \emph{pure}.
  A sequence is automatic if it is $k$-automatic for some $k$. 
\end{definition}
There is nothing special about reading the input from left to right.
Indeed, given a $k$-automatic sequence $a(n)$ there exists a $k$-DFAO with \emph{reverse reading} producing 
the sequence $a(n)$, i.e. this $k$-DFAO reads the input $(n)_k$ from right to left (this can be found for example in~\cite[Theorem 4.3.3]{Allouche2003}).

 Let us  recall how one can change the input alphabet $\Sigma_k$ to $\Sigma_{k^{\ell}} = \{0,\ldots,k^{\ell}-1\}$.
 
  \begin{lemma}
  	Let $\cA = (Q, \Sigma_k, \delta, q_0, \Delta,\tau)$ be a $k$-DFAO such that $\delta(q_0,0) = q_0$.
  	Then, for every integer $\ell\geq 1$,  the $k^\ell$-DFAO $\cA' = (Q, \Sigma_{k^\ell}, \delta, q_0, \Delta,\tau)$ produces the same automatic sequence. 
  \end{lemma}

  \begin{proof}
  	This follows directly from the extension of $\delta$ to $Q \times \Sigma^{*} \to Q$, the way the representation in base $k$ and in base $k^{\ell}$ correspond to each other and that $\delta(q_0,0) = q_0$ allows us to ignore leading zeros both for $\cA$ and $\cA'$.
  \end{proof}


%
%

\subsection{Densities for automatic sequences}

We recall in this section some results about densities and logarithmic densities for automatic sequences.

  \begin{lemma}[Theorem 7 in \cite{Cobham1972}]
	Let $a(n)$ be an automatic sequence. Then the logarithmic density exists for every $\alpha$, i.e.
	\begin{align*}
		\lim_{N\to \infty}\frac{1}{\log N} \sum_{n\leq N} \frac{1}{n} \ind_{[a(n) = \alpha]},
	\end{align*}
	exists and is denoted by $d_{\log}(a(n), \alpha)$.
\end{lemma}

\begin{lemma}\label{le_log0_to_0}
	Let $a(n)$ be an automatic sequence, such that the logarithmic density of $\alpha$ is $0$, 
	then the density of $\alpha$ exists and equals $0$.
\end{lemma}

\begin{proof}
	It follows directly by partial summation that for any sequence $a$, we have
	\begin{align*}
		\liminf_{N\to\infty} \frac{1}{N} \sum_{n\leq N} \ind_{[a(n)=\alpha]} \leq \liminf_{N\to\infty} \frac{1}{\log N} \sum_{n\leq N} \frac{1}{n} \ind_{[a(n)=\alpha]}.
	\end{align*}
	By assumption we know that the logarithmic density is $0$ and, therefore, the lower density of $\alpha$ is $0$.
	By \cite[Theorem 11]{Cobham1972} we know that this can only be the case if the density is $0$.
\end{proof}
  
  \begin{lemma}\label{le_0_to_banach0}
  	Let $a(n)$ be an automatic sequence, such that the density of $\alpha$ is $0$,  then the upper Banach density\footnote{The upper Banach density of $a$ is defined as $d^{*}(a) := \limsup \limits_{N-M \to \infty} \frac{\#\{M \leq n \leq N: a(n) \neq 0\}}{N-M+1}$.} is also $0$.
  \end{lemma}
  
  \begin{proof}
  	The set of integers for which $a(n) = \alpha$ is contained in a set with a missing digit by \cite[Theorem 9]{Cobham1972}. This immediately implies the statement.
  \end{proof}
  Lemma~\ref{le_log0_to_0} and Lemma~\ref{le_0_to_banach0} tell us that some (in general quite different) notions of sparseness actually coincide for automatic sequences.

 \subsection{Some subclasses of automata and automatic sequences}

 In this section, we recall various definitions about automata and automatic sequences. 
 
 \begin{definition}\label{def_minimal}
	A $k$-DFAO $\cA = (Q,\Sigma_k,\delta,q_0, \Delta, \tau)$ and the corresponding automatic sequence 
	is called \emph{minimal} if 
	\begin{itemize}
		\item For every $q \in Q$ there exists $\bfw \in \Sigma_k^{*}$ such that $\delta(q_0, \bfw) = q$.
		\item For every two different states $q_1, q_2 \in Q$ there exists $\bfw \in \Sigma_k^{*}$ such that $\tau(\delta(q_1, \bfw)) \neq \tau(\delta(q_2, \bfw))$.
	\end{itemize}
\end{definition}

\begin{fact}[\cite{Allouche2003} Corollary 4.1.9]
	Any $k$-automatic sequence can be produced by a minimal $k$-DFAO.
\end{fact}

   \begin{definition}
  	A $k$-DFA $\mathcal{A} = (Q, q_0, \Sigma_k, \delta)$ is \emph{strongly connected} if for any $q_1, q_2 \in Q$ there exists $\bfw \in \Sigma^{*}$ such that $\delta(q_1, \bfw) = q_2$. It is  \emph{primitive} if there exists some $\ell \in \N$ such that for any $q_1, q_2 \in Q$ there exists $\bfw \in \Sigma^{\ell}$ such that $\delta(q_1, \bfw) = q_2$. 
 Finally,  $\mathcal{A}$ is \emph{prolongable} if $\delta(q_0, 0) = q_0$.
  \end{definition}
  
    \begin{definition}
  	A $k$-automatic sequence is said to be \emph{prolongable} (resp.\ \emph{primitive}) if it can be produced by a $k$-DFAO whose corresponding $k$-DFA is prolongable (resp.\ primitive). It is called \emph{pure} if it can be produced by a $k$-DFAO whose 
output function is the identity.  
  \end{definition}

  \begin{lemma}\label{le_str_connected_primitive}
  	Let $\cA= (Q, q_0, \Sigma_k, \delta)$ be a strongly connected $k$-DFA 
	such that there exist some $q \in Q$ and $ i \in \Sigma$ with $\delta(q,i) = q$. Then $\cA$ is primitive.
  \end{lemma}
  
  \begin{proof}
  	Let $q_1, q_2 \in Q$. As $\cA$ is strongly connected, there exist $\bfw_1, \bfw_2 \in \Sigma^{*}$ such that $\delta(q_1, \bfw_1) = q, \delta(q, \bfw_2) = q_2$. Thus, we find that $\delta(q_1, \bfw_1 i^{n} \bfw_2) = q_2$ for any $n \in \N$. This shows that for any sufficiently large $\ell$ there exists some $\bfw \in \Sigma^{\ell}$ such that $\delta(q_1,\bfw) = q_2$. As this works for all (finitely many) pairs $q_1,q_2 \in Q$ we find some $\ell$ that works for all pairs simultaneously.
  \end{proof}

  \begin{definition}\label{def_minset}
  	Let $\cA = (Q, q_0, \Sigma_k, \delta)$ be a $k$-DFA.
  	A \emph{final component} of $\mathcal{A}$ is a minimal (with respect to inclusion) non-empty set $F \subseteq Q$ that is closed under $\delta(.,.)$.
  	The \emph{column number} of $\mathcal{A}$ is defined by
  	\begin{align*}
  		c(\mathcal{A}) := \min_{w \in \Sigma^{*}} \abs{\delta(Q, w)}.
  	\end{align*} 
  	We define $\mathcal{X}(\mathcal{A})$ as the set of subsets of $Q^{c}$ that are realized as $\delta(Q, w)$ for some $w\in \Sigma^{*}$.
  	Furthermore, we call a word $w$ \emph{minimizing} if $\abs{\delta(Q, w)} = c(\cA)$. If $c(\cA) = 1$, we call it \emph{synchronizing}.
  \end{definition}

\section{ A structural result for automatic sequences}

  This section is dedicated to the following structural result  concerning automatic sequences.
  
\begin{proposition} \label{prop:decomposition}
  Let $(a(n))_{n\geq 0}$ be a $k$-automatic sequence.
  Then there exists a finite set $\mathcal{B} = \{b_1,b_2,\ldots, b_s\}$ of $k$-automatic sequences that 
  are produced by some prolongable and primitive $k^\ell$-DFAO, where $\ell\geq 1$ is an integer, 
 and  with the following property.
  For every $b_i = (b_i(n))_{n\in \N} \in \mathcal{B}$, we set  
  $$M_i := \{m \in \N: a(mk^{\lambda} + r) = b_i(mk^{\lambda} +r) \,,
  \;\forall \lambda \in \mathbb{N}, 0\leq r < k^{\lambda}\} \,.$$
  The sets $M_i$,  $1 \leq i \leq s$, are pairwise disjoint and the logarithmic densities of $M_i$,  $1 \leq i \leq s$, 
  exist and are positive.
  Furthermore, the (upper Banach) density of $M_0 := \mathbb{N} \setminus \cup_i M_i$ exists and 
  equals $0$.
\end{proposition}

This proposition will allow us to approximate an automatic sequence $(a(n))_{n\in \mathbb{N}}$ by the 
primitive automatic sequences $(b_i(n))_{n\in \mathbb{N}}$. 
\footnote{A quite similar result can be found in~\cite{Byszewski2019}. They show that the minimal components of a subshift corresponding to a $k$-automatic sequence are given by primitive and prolongable $k^{\ell}$-automatic sequences. This allows us to cover the sequence $(a(n))_{n\geq 0}$ by arbitrary shifts of these finitely many $k^{\ell}$-automatic sequences. Since we are ultimately interested in (possibly) sparse subsequences, we need to avoid these shifts. This is exactly achieved by Proposition~\ref{prop:decomposition}, while also giving some information about how each sequence $b_i$ covers $a$.}
We start by proving an auxiliary result, 
which shows that the $M_i$ are $k$-automatic sets, i.e. the indicator function is $k$-automatic. 

\begin{lemma}\label{le_M_i_automatic}
	Let $k \geq 2$ and $(a(n))_{n\geq 0}, (b(n))_{n\geq 0}$ be $k$-automatic sequences. Then so is $(c(n))_{n\geq 0}$, where
	\begin{align*}
		c(n) = \left\{\begin{array}{cl} 1, & \mbox{if } a(nk^{\lambda}+s) = b(nk^{\lambda}+s) \mbox{ for all } \lambda \geq 0, 0\leq s < k^{\lambda}\\ 0, & \mbox{otherwise.} \end{array}\right. 
	\end{align*}
\end{lemma}

Before proving Lemma \ref{le_M_i_automatic}, we recall the following definition. 

 \begin{definition}\label{def_product_automata}
  	Let $\cA^{(1)} = (Q^{(1)}, \Sigma_k, \delta^{(1)}, q_0^{(1)}), \cA^{(2)} = (Q^{(2)}, \Sigma_k, \delta^{(2)}, q_0^{(2)})$ be two $k$-DFA. Then $\cA = (Q^{(1)} \times Q^{(2)}, \Sigma_k, \delta, q_0)$ is a $k$-DFA that we call the \emph{product} of $\cA^{(1)}$ and $\cA^{(2)}$, where $\delta = \delta^{(1)} \times \delta^{(2)}, q_0 = (q_0^{(1)}, q_0^{(2)})$, i.e. $\delta((q^{(1)}, q^{(2)}), \bfw) = (\delta^{(1)}(q^{(1)}, \bfw), \delta^{(2)}(q^{(2)}, \bfw))$. 
  \end{definition}

\begin{proof}[Proof of Lemma \ref{le_M_i_automatic}] 
     Let $\cA^{(1)} = (Q^{(1)}, \Sigma_k, \delta^{(1)}, q_0^{(1)},\Delta^{(1)},\tau^{(1)})$ denote a minimal $k$-DFAO with reverse reading that produces the sequence $a(n)$, and let $\cA^{(2)} = (Q^{(2)}, \Sigma_k, \delta^{(2)}, q_0^{(2)},\Delta^{(2)},\tau^{(2)})$ denote a minimal $k$-DFAO with reverse reading that produces the sequence $b(n)$.  
     For every $q\in Q^{(1)}$, we let $a_q(n)$ (resp.\ $b_q(n)$) denote the sequence produced by $\cA^{(1)}$ 
     (resp.\ $\cA^{(2)}$) when replacing the initial state by $q$. 
     
     For every pair $(q_1,q_2)\in Q^{(1)}\times Q^{(2)}$, we define the sequence $c_{q_1,q_2}(n)$ by 
   \begin{align*}
		c_{q_1,q_2}(n) = \left\{\begin{array}{cl} 1, & \mbox{if } a_{q_1}(n) = b_{q_2}(n) \\ 0, & \mbox{otherwise.} \end{array}\right. 
	\end{align*}
Then $c_{q_1,q_2}(n)$ is $k$-automatic for it can be produced using the product of the $k$-DFA 
$\cA_{q_1}:= (Q^{(1)}, \Sigma_k, \delta^{(1)}, q_1)$ and  $\cA_{q_2}:=(Q^{(2)}, \Sigma_k, \delta^{(2)}, q_2)$ endowed with the output function $\tau$ defined by $\tau(q,p)=1$ if $\tau^{(1)}(q)=\tau^{(2)}(p)$, and $\tau(q,p)=0$ otherwise. 

Now, setting 
$$
\mathcal S := \{(q_1,q_2)\in Q^{(1)}\times Q^{(2)} : \exists \bfw \in \Sigma_k^* \mbox{ such that } \delta^{(1)}(q_0^{(1)},\bfw)=q_1 \mbox{ and } 
\delta^{(2)}(q_0^{(2)},\bfw)=q_2\} \,,
$$
we get that 
$$
c(n)=\prod_{(q_1,q_2)\in \mathcal S} c_{q_1,q_2}(n) \,.
$$
Hence $c(n)$ is $k$-automatic as a finite product of integer-valued $k$-automatic sequences. 
\end{proof}

\begin{remark}\label{re_i,j}
	The stated Proposition is in a form that is oriented towards applicability. However, for the proof 
	we will use a different description of $\mathcal{B}$, 
	i.e., $\mathcal{B} = \{b_{i,j}: 1\leq i \leq r, 1\leq j \leq c_i\}$, 
	where $r$ denotes the number of different final components and $c_i$ the column number of the 
	$i$-th final component.
\end{remark}

\begin{proof}[Proof of Proposition~\ref{prop:decomposition}]
We start by noting that by Lemma~\ref{le_M_i_automatic}, the indicator functions of the $M_i$ are automatic and, 
thus, the logarithmic densities of the $M_i$ exist. Assume now that $m \in M_{i}$ and $k^{\ell-1} \leq m < k^{\ell}$, 
then we have that $m k^{r} + n \in M_{i}$ for all $r \in \N, n < k^{r}$. Hence a simple computation shows that the logarithmic density of $M_{i}$ is at least $\frac{\log(m+1) - \log(m)}{\log(k^{r})}$.
Thus it only remains to show that we can choose the $b_i$ in such a way that the $M_i, 1\leq i \leq s$ are disjoint and $M_0$ has upper Banach density $0$.

 We can assume without loss of generality that $(a(n))_{n\geq 0}$ is minimal and for any $q \in Q$,
  \begin{align}\label{eq_contract_loops}
  	(\exists n \in \N: \delta(q,0^n) = q) \qquad \Rightarrow \qquad (\delta(q,0) = q),
  \end{align}
  as we can change $\Sigma_k = \{0,\ldots,k-1\}$ to $\Sigma_{k^\ell} = \{0,\ldots,k^{\ell}-1\}$. We consider now the final components which we call $F_1, F_2, \ldots, F_r$. First we claim that for any final component $F_i$ there exists some set $M_0^{(i)} \in \mathcal{X}(F_i)$\footnote{The set $\mathcal{X}$ was introduced in Definition~\ref{def_minset}.} such that every element of $M_0^{(i)}$ is fixed under $\delta(.,0)$. 
  Fix any $i$, $1 \leq i \leq r$, and take some set $M \in \mathcal{X}(F_i)$. We consider now $M_j := \delta(M, 0^j)$, and one sees easily that $M_j \in \mathcal{X}(F_i)$. As $\mathcal{X}(F_i)$ is finite, there exists some $M_0^{(i)}$ for which there exists some $\ell$ with $\delta(M_0^{(i)}, 0^{\ell}) = M_0^{(i)}$. Thus we see that $\delta(.,0^{\ell})$ is a bijection from $M_0^{(i)}$ to itself. Therefore, we know that a properly chosen power is the identity, i.e. $\delta(q,0^{\ell'})=q$ for all $q \in M_0^{(i)}$ and the claim follows by \eqref{eq_contract_loops}.
  
  This shows in particular that all the $F_i$ are primitive by Lemma~\ref{le_str_connected_primitive}. 
  Now we are able to define the sequences $b_{i,j}(n)$. For every final component $F_i$ we define 
  $c_i := c(F_i)$ many different automatic sequences corresponding to the automata 
  $(F_i, q_0^{(i,j)}, \Sigma_{k^{\ell'}}, \delta\restriction_{F_i}, \tau\restriction_{F_i})$ 
  for every $q_0^{(i,j)} \in M_0^{(i)}$. We call the corresponding automatic sequences $b_{i,j}(n)$ 
  and the corresponding automata $\mathcal{B}_{i,j}$. 
  We note that $\delta(q_0, (m)_k) = \delta(q_0^{(i,j)}, (m)_k)$ 
  if and only if $m \in M_{i,j}$, by the minimality of $(a(n))_{n\geq 0}$. 
Thus we see directly that all the $M_{i,j}$ have to be disjoint. Indeed, 
let us assume that $m \in M_{i_1, j_1} \cap M_{i_2, j_2}$, which can only happen if 
$\delta(q_0^{(i_1, j_1)}, (m)_k) = \delta(q_0^{(i_2, j_2)}, (m)_k)$. This can clearly be only the case if 
$i_1 = i_2 =: i$, as the final components are disjoint. 
However, this would also imply that $\abs{\delta(M_0^{(i)}, (m)_k)} < c(F_i)$ which gives a contradiction.

 It only remains to show that the (upper Banach) density of $M_0$ equals $0$.
We find by~\cite[Lemma 3.1]{Byszewski2020} that there exists a word $\bfw_1 \in \Sigma_k^{*}$ such that if $\bfv \in \Sigma_k^{*}$ contains $\bfw_1$ as a factor then $\delta(q_0, \bfv)$ belongs to a strongly connected component of $\mathcal{A}$, i.e. one of the $F_i$.
Next we find a word $\bfw_2$ that is minimizing for all the $F_i$. Therefore, we can take for example the concatenation of words that are minimizing for a single $F_i$.
Next we aim to show that if $\bfv \in \Sigma_k^{*}$ contains $\bfw := \bfw_1 \bfw_2$ as a subword, 
then there exists $i,j$ such that $\delta(q_0, \bfv) = \delta(q_0^{(i,j)}, \bfv)$, i.e. $[\bfv]_k \notin M_0$.
We note that we can split $\bfv = \bfv_1 \bfv_2$ such that $\bfw_1$ is a subword of $\bfv_1$ and $\bfw_2$ is a 
subword of $\bfv_2$.
The defining property of $\bfw_1$ ensures that there exists some $i$ such that $\delta(q_0, \bfv_1) \in F_i$.
As $\bfw_2$ (and therefore also $\bfv_2$) is minimizing for $F_i$, we have $M^{(i)} := \delta(F_i, \bfv_2) \in \mathcal{X}(F_i)$ such that $\delta(q_0, \bfv) \in M^{(i)}$. Moreover, we find by the properties of $\mathcal{X}(F_i)$ that $\delta(M_0^{(i)}, \bfv) \in \mathcal{X}(F_i)$ and, therefore, $\delta(M_0^{(i)}, \bfv) = M^{(i)}$. Thus, there exists $q_0^{(i,j)} \in M_0^{(i)}$ such that $\delta(q_0, \bfv) = \delta(q_0^{(i,j)}, \bfv)$.

Thus $M_0$ is contained in a set with a missing digit 
(i.e. the set of numbers whose base $k^{\abs{\bfv}}$ expansions do not contain any occurrences of $\bfv$) and, thus,  
its upper Banach density is $0$.
\end{proof}

\begin{remark}\label{re_decomposition_b_i}
	We discuss here shortly how to determine the $b_{i,j}$ and the $M_{i,j}$ that appear in the proof of Proposition~\ref{prop:decomposition} (also recall Remark~\ref{re_i,j}).
	Given a pure $k$-automatic sequence $a(n)$ with corresponding automaton $\mathcal{A} = (Q, \Sigma_k, \delta, q_0)$.
	We first ensure that~\eqref{eq_contract_loops} holds by possibly changing $\Sigma_k$ to $\Sigma_{k^{\ell}}$ for some $\ell \geq 1$.
	Then we determine the final components $F_i, 1\leq i \leq s$ of $\mathcal{A}$.
	The proof of Proposition~\ref{prop:decomposition} ensures that for every $F_i$ there exists some $M_0^{(i)} \in \mathcal{X}(F_i)$ such that every element of $M_0^{(i)}$ is fixed by $\delta(.,0)$.
	This allows us to define the $b_{i,j}$ as the automatic sequence corresponding to the automaton $(Q, \Sigma, \delta, q_0^{(i,j)})$ where $q_0^{(i,j)} \in M_0^{(i)}$.
	
	We have seen in the proof of Proposition~\ref{prop:decomposition} that 
	\begin{align*}
		m \in M_{i,j} \text{ if and only if } 
		\delta(q_0, (m)_k) = \delta(q_0^{(i,j)}, (m)_k).
	\end{align*}
	Thus, we can actually just consider the $k$-automatic sequence $((a(n), b_{i,j}(n)))_{n\geq 0}$ (see Definition~\ref{def_product_automata}) and see that the indicator function of $M_{i,j}$ is the just the projection of the previous sequence where $\tau((x,y)) = \ind_{[x = y]}$.
\end{remark}

  \section{Transfer of densities}\label{sec:transfer}
  In this section we prove Theorem~\ref{Th1} which allows to compute the logarithmic density of a general automatic sequence along a subsequence $(n_{\ell})_{\ell\in\N}$ when knowing the density of primitive automatic sequences along the same subsequence. The main ingredient is the structural result we discussed in the previous section, Proposition~\ref{prop:decomposition}.
  Furthermore, it is in this context very useful to use summation by parts.
 
  \begin{lemma}\label{le_partial_summation}
  	Let $(a_n), (b_n)$ be two sequences of complex numbers.
  	Then
  	\begin{align*}
  		\sum_{n=0}^{N} a_n b_n = b_{N} \sum_{n=0}^{N} a_n + \sum_{n=0}^{N-1} (b_n - b_{n+1}) \sum_{\ell = 0}^{n} a_{\ell} .
  	\end{align*}
  \end{lemma}
 
  We also need the estimate 
  \begin{align}\label{eq_harmonic_sum}
  	\abs{\sum_{\ell = a+1}^{b} \frac{1}{\ell} - \log\rb{\frac{b}{a}}} \leq \frac{1}{a} \,,
  \end{align}
  which follows from comparing the sum with $\int_{a}^{b} \frac{1}{x} dx$.
To prove Theorem~\ref{Th1}, we are interested in computing
\begin{align*}
  \lim_{L \to \infty} \frac{1}{\log(L)}\sum_{\ell<L} \frac{1}{\ell} \inda = \lim_{N \to \infty} \frac{1}{\log(g(N))} \sum_{\ell: n_{\ell}\leq N} \frac{1}{\ell} \inda.
\end{align*}

  Proposition~\ref{prop:decomposition} allows us to estimate parts of the sum.  Fix some $m<k^{\nu-\lambda}, m \in M_i$, 
  we find
  \begin{align*}
    \sum_{\substack{\ell<L\\ mk^{\lambda}\leq n_\ell < (m+1) k^{\lambda}}} \frac{1}{\ell} \inda = \sum_{\substack{\ell<L\\ mk^{\lambda}\leq n_\ell < (m+1) k^{\lambda}}} \frac{1}{\ell} \indbi.
  \end{align*}
  
  \begin{lemma}\label{le:density}
    Let $(b(n_{\ell}))_{\ell\geq 0}$ be such that the density of $\alpha$ exists, i.e.
    \begin{align*}
      \lim_{L\to\infty} \frac{1}{L} \sum_{\ell \leq L} \indb = d_b(\alpha).
    \end{align*}
    Then,
    \begin{align*}
    	\sum_{x\leq \ell \leq y} \frac{1}{\ell} \indb = \log(y/x) d_b(\alpha) + o_{x \to \infty}(1 + \log(y/x)).
    \end{align*}
  \end{lemma}  
  \begin{proof}
  Fix $\varepsilon>0$ and let $x_0$ be large enough, such that for any $x\geq x_0$ we have
    \begin{align*}
      \abs{\frac{1}{x}\sum_{\ell \leq x} \indb - d_b(\alpha)} \leq \varepsilon.
    \end{align*}
    By partial summation, i.e. Lemma~\ref{le_partial_summation}, we find 
    \begin{align*}
      \sum_{x\leq \ell \leq y} c_{\ell} \frac{1}{\ell} &= \sum_{x\leq \ell \leq y} c_{\ell} \frac{1}{y} + \sum_{x\leq n < y} \sum_{x\leq \ell \leq n} c_{\ell} \rb{\frac{1}{n} - \frac{1}{n+1}}\\
        &= \frac{1}{y} \sum_{x\leq \ell \leq y} c_{\ell} + \sum_{x\leq n < y} \frac{1}{n(n+1)} \sum_{x\leq \ell \leq n} c_{\ell}.
    \end{align*}
	This gives
	\begin{align*}
		\abs{\sum_{x\leq \ell \leq y} \frac{1}{\ell} (\indb - d_{b}(\alpha))}&\leq \frac{1}{y} \abs{\sum_{x\leq \ell \leq y} \indb - d_{b}(\alpha)}\\
			&\qquad + \sum_{x\leq n < y} \frac{1}{n(n+1)}\abs{\sum_{x\leq \ell \leq n} \indb - d_b(\alpha)}\\
			&\leq \varepsilon \frac{x-1+y}{y} + \varepsilon \sum_{x\leq n < y} \frac{x-1+n}{n(n+1)}\\
			&\leq 2 \varepsilon + 2 \varepsilon \sum_{x< \ell \leq y} \frac{1}{\ell} \cdot
	\end{align*}
	The result follows by applying~\eqref{eq_harmonic_sum} twice.
  \end{proof}

  \begin{proof}[Proof of (i) of Theorem~\ref{Th1}]
  Naturally, we are interested in estimating
  \begin{align}\label{eq_d_log_def}
  	d_{\log}(a(n_{\ell}), \alpha) = \lim_{L \to \infty} \frac{1}{\log(L)} \sum_{1\leq \ell \leq L}\frac{1}{\ell} \inda \,.
  \end{align}
  Actually, we aim to show, with the notation from Proposition~\ref{prop:decomposition}, that 
  \begin{align}\label{eq_d_log}
  	d_{\log}(a(n_{\ell}), \alpha) = \sum_{1\leq i \leq s} d_{\log}(M_i) \cdot d(b_i(n_{\ell}), \alpha).
  \end{align}
    
    We note that the limit in \eqref{eq_d_log_def} is invariant under multiplying $L$ by a bounded constant. 
    This means, it is sufficient to consider only a subsequence $(L_{\nu})_{\nu \in \N}$, where $L_{\nu+1}/L_{\nu} $ is bounded.  
    In particular, we can choose $L_{\nu} = g(k^{\lambda \nu})$, as $g(k^{\lambda (\nu+1)})/g(k^{\lambda \nu}) \to k^{\lambda \beta}$ for any $\lambda \in \N_{\geq 1}$.
    Moreover, Lemma~\ref{le_log_g} shows that we can replace $\log(g(k^{\lambda \nu}))$ by $\log(k^{\beta \lambda \nu})$ in \eqref{eq_d_log_def}.
   
   On the other hand,  we find
    \begin{align*}
      \frac{1}{\log(k^{\beta \lambda \nu})} \sum_{\ell \in \mathbb{N}: n_{\ell} < k^{\lambda \nu}} \frac{1}{\ell} \inda 
      &= \frac{1}{\log(k^{\beta \lambda \nu})} \sum_{1\leq m <k^{\lambda}} \sum_{j<\nu} \sum_{\substack{\ell \in \mathbb{N}\\mk^{\lambda j} \leq n_{\ell} < (m+1)k^{\lambda j}}} \frac{1}{\ell} \inda\\
      &= \frac{1}{\log(k^{\beta \lambda \nu})} \sum_{0 \leq i \leq s} \sum_{\substack{1\leq m <k^{\lambda}\\ m \in M_i}} \sum_{j<\nu} \sum_{\substack{g(mk^{\lambda j}) \leq \ell < g((m+1)k^{\lambda j})}} \frac{1}{\ell} \inda \,.
    \end{align*}
  Thus we are interested in computing 
  \begin{align*}
    \sum_{\substack{\ell \in \mathbb{N}\\mk^{\lambda j} \leq n_{\ell} < (m+1)k^{\lambda j}}} \frac{1}{\ell} \indbi = \sum_{g(mk^{\lambda j}) < \ell \leq g((m+1)k^{\lambda j})} \frac{1}{\ell} \indbi
  \end{align*}
  for $m \in M_i$. By applying Lemma~\ref{le:density} and Lemma~\ref{le_log_g} we find,
  \begin{align*}
  	\sum_{\ell = g(mk^{\lambda j})+1}^{g((m+1)k^{\lambda j})} \frac{1}{\ell} \indbi = \beta \log\rb{1 + \frac{1}{m}} \cdot d(b_i(n_{\ell}), \alpha)  +  o_{g(m k^{\lambda j}) \to \infty}\rb{ 1 + \beta \log\rb{1 + \frac{1}{m}} }.
  \end{align*}
	As $m \geq 1$, we have that $\log\rb{1 + \frac{1}{m}}$ is bounded.
	Furthermore, $m k^{\lambda j} \geq k^j$, which allows us to replace $g(m k^{j\lambda}) \to \infty$ by $j\to \infty$.
	This gives in total,
	\begin{align*}
		\sum_{\ell = g(mk^{\lambda j})+1}^{g((m+1)k^{\lambda j})} \frac{1}{\ell} \indbi = \beta \log\rb{1 + \frac{1}{m}} \cdot d(b_i(n_{\ell}), \alpha) + o_{j\to \infty}(1).
	\end{align*}
 
 A simple computation yields,
  \begin{align*}
    \abs{\log\left(1 + \frac{1}{m}\right) - \frac{1}{m}} \leq \frac{1}{m^2},
  \end{align*}
  which gives for every $1\leq m \leq k^{\lambda}$,
  \begin{align*}
    \sum_{\ell = g(mk^{\lambda j})+1}^{g((m+1)k^{\lambda j})} \frac{1}{\ell} \indbi = \beta \frac{1}{m} \cdot d(b_i(n_{\ell}), \alpha) + O^{*}\rb{\frac{1}{m^2}} + o_{j \to \infty}(1).
  \end{align*}
  This gives
  \begin{align*}
    \sum_{\substack{1\leq m <k^{\lambda}\\ m \in M_i}} &\sum_{j<\nu} \sum_{\substack{\ell \in \mathbb{N}\\mk^{\lambda j} \leq n_{\ell} < (m+1)k^{\lambda j}}} \frac{1}{\ell} \indbi\\
     &= \sum_{\substack{1\leq m <k^{\lambda}\\ m \in M_i}} \sum_{j<\nu} \rb{d(b_i(n_{\ell}), \alpha) \frac{\beta}{m} + O^{*}\rb{\frac{1}{m^2}} + o_{j \to\infty}(1)}\\
       	&= \nu \cdot d(b_i(n_{\ell}), \alpha) \cdot \rb{\sum_{\substack{1\leq m < k^{\lambda}\\m \in M_i}} \frac{\beta}{m} + O^{*}\rb{\frac{1}{m^2}}} + o_{\nu \to \infty} (\nu k^\lambda).
  \end{align*}
  We note that
  \begin{align*}
  	\sum_{\substack{1 \leq m \leq k^{\lambda}\\ m \in M_i}} \frac{1}{m} = \lambda \log(k) d_{\log}(M_i) + o_{\lambda \to \infty} (\lambda),
  \end{align*}
  and as $\sum 1/m^2$ is bounded we have in total
  \begin{align*}
  	\sum_{\substack{1\leq m <k^{\lambda}\\ m \in M_i}} &\sum_{j<\nu} \sum_{\substack{\ell \in \mathbb{N}\\mk^{\lambda j} \leq n_{\ell} < (m+1)k^{\lambda j}}} \frac{1}{\ell} \indbi\\
     &= \nu \lambda \log(k) \cdot \beta \cdot d_{\log}(M_i) d(b_i(n_{\ell}), \alpha) + o_{\lambda \to \infty} (\nu \lambda) + o_{\nu \to \infty} (\nu k^\lambda).
  \end{align*}
  We note that for any $\lambda \in \N$ we have $o_{\nu \to \infty} (\nu k^{\lambda}) = o_{\nu \to \infty}(\nu \lambda)$.
  Thus, letting first $\nu \to \infty$ and then $\lambda \to \infty$ gives \eqref{eq_d_log}, as wanted.
    \end{proof}

  \begin{proof}[Proof of (ii) of Theorem~\ref{Th1}]
  	Consider some large $\lambda$ and
  	define for $N \geq k^{\lambda}$ an integer $\nu$ such that $k^{\nu+\lambda-1} \leq N < k^{\nu+\lambda}$ and $m_0 \in [k^{\lambda-1},k^{\lambda}]$ such that $m_0 k^{\nu} \leq N < (m_0+1) k^{\nu}$.

  	We are interested in computing
  	\begin{align}\label{eq_density_a}
  		\lim_{N\to \infty}\frac{1}{g(N)} \sum_{\ell \leq g(N)} \inda.
  	\end{align}
  	Changing $N$ to $m_0 k^{\nu}$ changes the limit in two ways.
  	The first contribution is due to the shortening of the sum and the second contribution is due to the changing of the normalizing factor. Both contributions change the value by at most
  	\begin{align*}
  		\frac{g(N) - g(m_0 k^{\nu})}{g(N)} \leq 1 - \frac{g(m_0 k^{\nu})}{g((m_0 + 1) k^{\nu})} \to_{\nu \to \infty} 1 - \rb{1 - \frac{1}{m_0 + 1}}^{\beta} \leq \frac{1}{m_0 + 1} \leq \frac{1}{k^{\lambda-1}}.
  	\end{align*}
  	
  	Therefore, we are interested in computing
  	\begin{align*}
  		\frac{1}{g(m_0 k^{\nu})} \sum_{m<m_0} \sum_{g(mk^{\nu}) \leq \ell < g((m+1) k^{\nu}} \inda.
  	\end{align*}
  	We note that for any $\lambda$, each sum between $g(m k^{\nu})$ and $g((m+1) k^{\nu})$ is of substantial length for large enough $\nu$:
  	\begin{align*}
  		\lim_{\nu \to \infty} &\frac{g((m+1) k^{\nu}) - g(m k^{\nu})}{g(m_0 k^{\nu})} = \rb{\frac{m+1}{m_0}}^{\beta} - \rb{\frac{m}{m_0}}^{\beta}\\ &\geq 1 - \rb{\frac{m_0 -1}{m_0}}^{\beta} \geq 1 - \rb{\frac{k^{\lambda}-1}{k^{\lambda}}}^{\beta} > 0.
  	\end{align*}
  	Hence, for $m \in M_i$ for some $1\leq i \leq s$ we find
  	\begin{align}\label{eq_density_interval}
  		\lim_{\nu \to \infty}\frac{1}{g(m_0 k^{\nu})} \sum_{g(mk^{\nu})\leq \ell < g((m+1)k^{\nu})} \indbi = \frac{(m+1)^{\beta} - m^{\beta}}{m_0^{\beta}} d(b_i(n_{\ell}),\alpha).
  	\end{align}
  	This already shows, that the limit in~\eqref{eq_density_a} can only exist if $d(b_i(n_{\ell}), \alpha)$ coincides for all $1\leq i \leq s$.
  	For the other direction we need to show that the contribution of $m \in M_0$ is negligible. We find that
  	\begin{align*}
  		\lim_{\nu \to \infty} \frac{1}{g(m_0 k^{\nu})} \sum_{\substack{m \in M_0\\m < m_0}} g((m+1)k^{\nu}) - g(m k^{\nu})
  			= \sum_{\substack{m \in M_0\\m < m_0}} \frac{(m+1)^{\beta} - m^{\beta}}{m_0^{\beta}}.
  	\end{align*}
  	By the mean value theorem, there exists some $\xi \in [0,1]$ such that $(m+1)^{\beta} - m^{\beta} = \beta (m+\xi)^{\beta-1}$, which is monotonically decreasing in $\xi$ (as $0<\beta \leq 1$ and, therefore, $\beta-1 <0$). Thus, we have $(m+1)^{\beta} - m^{\beta} \leq \beta m^{\beta-1}$ which gives
  	\begin{align*}
  		\sum_{\substack{m \in M_0\\m < m_0}} \frac{(m+1)^{\beta} - m^{\beta}}{m_0^{\beta}} &\leq \sum_{\substack{m \in M_0\\m < m_0}} \frac{\beta  m^{\beta-1}}{m_0^{\beta}}\\
			&\leq \sum_{m\leq \abs{M_0 \cap[0, m_0-1]}} \frac{\beta m^{\beta-1}}{m_0^{\beta}}\\
			&\leq \rb{\frac{\abs{M_0 \cap[0,m_0-1]} + 1}{m_0}}^{\beta},
  	\end{align*}
	where the last inequality is obtained by estimating the sum by an integral.
  	This shows that the contribution of $m \in M_0$ is negligible for large $\lambda$.
  	We note that for $N \to \infty$ we also have $\nu \to \infty$ which gives,
  	\begin{align*}
  		\lim_{N\to \infty} \frac{1}{g(N)} \sum_{\ell \leq g(N)} \inda &= \lim_{N\to \infty} \frac{1}{g(m_0 k^{\nu})} \sum_{\ell \leq g(m_0 k^{\nu})} \inda + O\rb{\frac{1}{k^{\lambda-1}}}\\
  			&= \lim_{N\to \infty} \sum_{0 \leq m < m_0} \frac{1}{g(m_0 k^{\nu})} \sum_{g(mk^{\nu}) \leq \ell < g((m+1) k^{\nu})} \inda + O\rb{\frac{1}{k^{\lambda-1}}}.
  	\end{align*}
  	Let us assume that $d(b_i(n_{\ell}), \alpha) = d(\alpha)$ for all $1\leq i \leq s$. Thus, we have for $m \notin M_0$,
  	\begin{align*}
  		\lim_{\nu\to \infty} \frac{1}{g(m_0 k^{\nu})} \sum_{g(mk^{\nu}) \leq \ell < g((m+1) k^{\nu})} \inda = \frac{g((m+1) k^{\nu}) - g(m k^{\nu})}{g(m_0 k^{\nu})} d(\alpha).
  	\end{align*}
  	This gives
  	\begin{align*}
  		\lim_{N\to \infty} \frac{1}{g(N)} \sum_{\ell \leq g(N)} \inda &= d(b_i(n_{\ell}), \alpha) + O\rb{\rb{\frac{\abs{M_0 \cap[0,m_0-1]}}{m_0}}^{\beta}} + O\rb{\frac{1}{k^{\lambda-1}}}.
  	\end{align*}
  	The error terms vanish for $\lambda \to \infty$ as $m_0 \geq k^{\lambda-1}$.
  	Thus one needs to let first $\lambda \to \infty$ and then $N \to \infty$ to achieve the desired result.

  \end{proof}

  \section{Subsequences of primitive automatic sequences}\label{sec_primitive}

We start this section by discussing a result by the last author~\cite{Muellner2017}, which allows to represent a $k$-automatic sequence $a(n)$ which is primitive and prolongable as a combination of an almost periodic sequence and a sequence that looks random in some ways.
This representation has the form
\begin{align}\label{eq_structure_a}
	a(n) = f(s(n), T(n)),
\end{align}
  where $s(n)$ is a pure synchronizing $k$-automatic sequence taking values in $Q^{(c)}$ for some $c \geq 1$ and $T(n)$ takes values in a finite group $G$ with the following property.
 For every $j < k$ and $q \in Q^{(c)}$ there exists $g_{j, q}\in G$ such that $T(n \cdot k + j) = T(n) \cdot g_{j, s(n)}$ holds for all $n \in \N$.
We see that $T$ takes a particularly simple form when $s$ is constant -- this corresponds to a so called invertible (sometimes also called bijective) automatic sequence.

\begin{example}\label{ex_aut}
	We consider the following automaton, with input alphabet $\{0,1\}$.

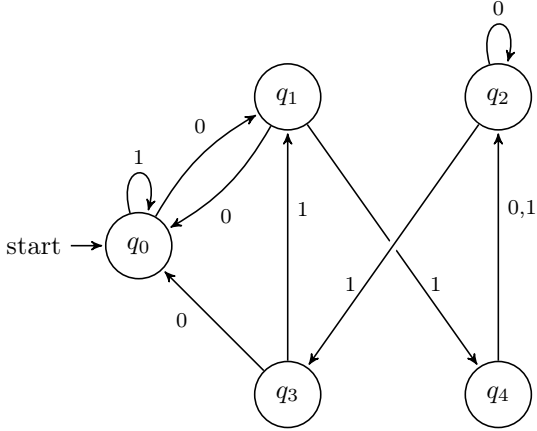
	 \begin{tikzpicture}[->,>=stealth',shorten >=1pt,auto,node distance=2.8cm, semithick, bend angle = 15, cross line/.style={preaction={draw=white,-,line width=4pt}}]
    
    \node[state, initial](A)                    {$q_0$};
    \node[state]         (B) [above right of=A] {$q_1$};
    \node[state]         (C) [right of=B] 	{$q_2$};
    \node[state]         (D) [below right of=A] {$q_3$};
    \node[state]         (E) [right of=D]       {$q_4$};
    
    \path [every node/.style={font=\footnotesize}, pos = 0.66]
    (A) edge [bend left]  node 		 {0}      (B)
	edge [loop above] node [pos=0.5] {1} 	(A)
    (B) edge [bend left]  node 		 {0}	(A)
	edge              node 	[right]	 {1}	(E)
    (C) edge [loop above] node [pos=0.5] {0} 	(C)
	edge [cross line] node 	[left]	 {1}	(D)
    (D) edge              node 		 {0}	(A)
	edge              node 	[right]	 {1}	(B)
    (E) edge              node [right]	 {0,1}	(C);
    
  \end{tikzpicture}
  
	The sequence $s(n)$ corresponds to the following automaton.

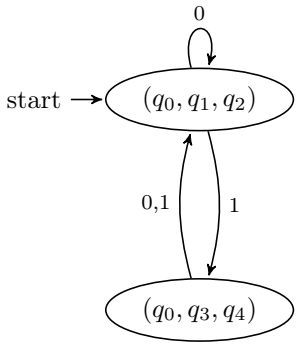
\tikzset{elliptic state/.style={draw,ellipse}}
\begin{tikzpicture}[->,>=stealth',shorten >=1pt,auto,node distance=2.8cm, semithick, bend angle = 15, cross line/.style={		preaction={draw=white,-,line width=4pt}}]
				
				\node[elliptic state, initial]	(A)                    {$(q_0, q_1, q_2)$};
				\node[elliptic state]	         (B) [below of=A] 	{$(q_0, q_3, q_4)$};
				
				\path [every node/.style={font=\footnotesize}, pos = 0.5]
				(A) edge [draw = black, loop above] node [black]		 	{0}      	(A)
						edge [draw = black, bend left]  node [black]		 	{1} 	(B)
				(B) edge [draw = black, bend left]  node [align=center]	{0,1}	(A);
			\end{tikzpicture}

	The group $G = S_3$ and the group elements $g_{j,q}$ are given by
  \begin{align*}
    g_{0,(q_0,q_1,q_2)} &= (12), \quad
    g_{1,(q_0,q_1,q_2)} = (23)\\
    g_{0,(q_0,q_3,q_4)} &= (12), \quad
    g_{1,(q_0,q_3,q_4)} = id,
  \end{align*}
	and the function $f$ is given by $f((q_{i_1}, q_{i_2}, q_{i_3}),g) = q_{i_{g^{-1}(1)}}$. For a more detailed treatment of this example see~\cite{Muellner2017}. 
\end{example}
  
  We start by discussing some properties of synchronizing automatic sequences. For a more detailed treatment of subsequences of synchronizing automatic sequences see~\cite{Deshouillers2015}.
  We recall that a word $\bfw \in \Sigma^{*}$ is synchronizing for an automaton $\cA = (Q, q_0, \Sigma, \delta, \Delta, \tau)$ if $\delta(q, \bfw) = \delta(q_0, \bfw)$ for all $q \in Q$.
  This implies directly that the concatenation of a synchronizing word with any word is again synchronizing.
  We define the set of synchronizing integers as follows.
  \begin{align*}
  	S := \{n \in \N: (n)_k \text{ is synchronizing}\}.
  \end{align*}
  We will also make use of a truncated version, $S_{\lambda} := S \cap [0,\ldots,k^{\lambda}-1]$.
  We recall that by the defining property of a synchronizing word, $s(n) = s(m)$ if $n \equiv m \bmod k^{\lambda}$ for $m \in S_{\lambda}$. Moreover, we have $\lim_{\lambda \to \infty} \frac{\abs{S_{\lambda}}}{k^{\lambda}} = 1$ by~\cite[Lemma 2.2]{Deshouillers2015}.
  This already shows that $s(n)$ is almost periodic, i.e. it can be (uniformly) approximated by periodic functions.
  
  $T(n)$, which is sometimes called the \emph{invertible part}, looks much more random in many ways, but still has some periodic properties. 
  In particular, there exists a normal subgroup $G_0$ such that $G/G_0 \cong \Z / d\Z$ for some $d \in \N$ which is coprime to $k$ and depends on the sequence $a(n)$. Furthermore, there exist cosets $G_0, G_1, \ldots, G_{d-1}$ such that $T(n) \in G_{(n \bmod d)}$ for all $n \in \N$.
  \\
  One of the key tools to study the distribution of sequences that take values in $G$ are (unitary and irreducible) representations (see for example~\cite{Serre1977} for more information on linear representations of finite groups). A $m$-dimensional unitary representation $D: G \to \U_m$ is a homomorphism from G to the set of unitary $m \times m$ matrices. It is said to be irreducible if there exists no non-trivial subspace $V \subset \C^{m \times 1}$ such that $D(g) \cdot V \subseteq V$ holds for all $g \in G$.
  The periodic behaviour described above manifests itself in the existence of special representations $D_{0}, D_{1}, \ldots, D_{d-1}$ form $G$ to $\U_1$ which can be defined via
  \begin{align*}
  	D_{j}(T(n)) = \e\rb{\frac{n \cdot  j}{d}}.
  \end{align*}
  We say that two representations $D, D'$ are equivalent if there exists a matrix $A \in U_m$ such that $D'(g) = A D(g) A^{-1}$ for all $g \in G$.
  It is a well-known fact that for a finite group $G$ there are only finitely many equivalence classes of irreducible and unitary representations. Furthermore, non-equivalent irreducible and unitary representations $D, D'$ are orthogonal, i.e.
  \begin{align*}
  	0 = \langle D, D' \rangle = \frac{1}{\abs{G}} \sum_{g\in G} D(g) \overline{D'(g)}.
  \end{align*}

  Very importantly, one can use representations to determine the asymptotic distribution of a sequence (see for example~\cite{Kuipers1974} for a proof).
  \begin{lemma}\label{le_distT}
  Let $G$ be a finite group and $\boldsymbol{\nu}$ be a probability measure on $G$. Then a sequence $(x_n)_{n \geq 0}$ is $\boldsymbol{\nu}$-uniformly distributed
  in $G$, i.e., $\frac{1}{N} \sum_{n < N}\delta_{x_n} \rightarrow \boldsymbol{\nu}$, if and only if 
  \begin{align}\label{eq_dist}
    \lim_{N\rightarrow \infty}\frac{1}{N} \sum_{n<N} D(x_n) = \int_{G} D \, d\boldsymbol{\nu}
  \end{align}
  holds for all irreducible unitary representations $D$ of $G$.
\end{lemma}
 \begin{remark}
 	The above lemma actually works in a much more general setting. That is, the same statement holds when $G$ is a compact group and $\boldsymbol{\nu}$ is a regular normed Borel measure in $G$.
 \end{remark}

  \subsection{A general strategy}
  Now we describe a method on how to work with subsequences of primitive automatic sequences using~\eqref{eq_structure_a}.
  We need another definition before tackling this task.
  \begin{definition}
  	A sequence $(n_{\ell})_{\ell \in \N}$ of nonnegative integers \emph{distributes regularly within residue classes} if for any $h \in \N, 0 \leq m < h$ there exists some $c_{n_{\ell}}(m; h)$ such that
  	\begin{align*}
  		\lim_{L \to \infty} \frac{\abs{\{\ell \leq L: n_{\ell} \equiv m \bmod h\}}}{L} = c_{n_{\ell}}(m; h)
  	\end{align*}
  	and it is multiplicative in the second argument, i.e. $c_{n_{\ell}}(m; h_1 \cdot h_2) = c_{n_{\ell}}(m; h_1) \cdot c_{n_{\ell}}(m; h_2)$ for any $m \in \N$ and co-prime $h_1, h_2$. We write $c(m; h) = c_{n_{\ell}}(m; h)$ if $n_{\ell}$ is clear from the context.
  \end{definition}
  Now we are able to state the main theorem of this subsection.
  \begin{theorem}\label{th_density_primitive}
  	Let $(n_{\ell})_{\ell \in \N}$ be a strictly increasing sequence that distributes regularly within residue classes such that
  	\begin{align}\label{eq_not_sparse}
  		\lim_{\lambda\to \infty} \sum_{\substack{m < k^{\lambda}\\m \notin S_{\lambda}}} c(m; k^{\lambda}) = 0.
  	\end{align}
  	Furthermore, we assume that for any irreducible and unitary representation $D$ different from $D_j$ and any $\lambda \in \N, m < k^{\lambda}$,
  	\begin{align}\label{eq_D_not_D_j}
  		\lim_{L \to \infty} \frac{1}{L} \sum_{\substack{\ell \leq L\\ n_{\ell} \equiv m \bmod k^{\lambda}}} D(T(n_{\ell})) = 0.
  	\end{align}
  	Then the densities of $s, T$ and $a$ along $n_{\ell}$ exist and are given by
  	\begin{align*}
  		d(s(n_{\ell}), q) &= \lim_{\lambda \to \infty} \sum_{m \in S_{\lambda}} \ind_{[s(m) = q]} c(m; k^{\lambda}) = \lim_{\lambda \to \infty} \sum_{m < k^{\lambda}} \ind_{[s(m) = q]} c(m; k^{\lambda})\\
  		d(T(n_{\ell}), g) &= \frac{d}{\abs{G}} \cdot c(j; d) \text{ for } g \in G_j\\
  		d(a(n_{\ell}), \alpha) &= \sum_{q \in Q, g \in G} d(s(n_{\ell}), q) \cdot d(T(n_{\ell}), g) \cdot \ind_{[f(q, g) = \alpha]}.
  	\end{align*}
  \end{theorem}
  
  \begin{proof}
  
  We first show that the limit
  \begin{align*}
  	\lim_{\lambda \to \infty} \sum_{r \in S_{\lambda}} \ind_{[s(r) = q]} c(r; k^{\lambda})
  \end{align*}
  indeed exists.
  We find directly that
  \begin{align*}
  	\sum_{r \in S_{\lambda}} \ind_{[s(r) = q]} c(r; k^{\lambda}) \leq \sum_{0\leq r < k^{\lambda}} c(r; k^{\lambda}) = 1.
  \end{align*}
  Let us now assume that $r \in S_{\lambda}$ with $s(r) = q$. Then we have that $r'k^{\lambda} + r \in S_{\lambda}$ and $s(r' k^{\lambda} + r) = q$ for all $0\leq r' < k$. Also, we have $\sum_{0\leq r' < k} c(r' k^{\lambda} + r; k^{\lambda+1}) = c(r; k^{\lambda})$.
This shows
\begin{align*}
	\sum_{r \in S_{\lambda}} \ind_{[s(r) = q]} c(r; k^{\lambda}) \leq \sum_{r \in S_{\lambda+1}} \ind_{[s(r) = q]} c(r; k^{\lambda+1}).
\end{align*}
Thus, we have a bounded and monotone sequence and, therefore, the limit exists.
  
  We work from now on only with $(a(n_{\ell}))_{\ell\in \N}$ and prove the other results along the way.
  We start by approximating $s$ by a periodic function,
  \begin{align*}
  	\frac{1}{L} \sum_{\ell \leq L} \inda &= \frac{1}{L} \sum_{m < k^{\lambda}} \sum_{\substack{\ell \leq L\\ n_{\ell} \equiv m \bmod k^{\lambda}}} \ind_{[f(s(n_{\ell}), T(n_{\ell})) = \alpha]}\\
  		&= \sum_{m \in S_{\lambda}} \frac{1}{L} \sum_{\substack{\ell \leq L\\ n_{\ell} \equiv m \bmod k^{\lambda}}} \ind_{[f(s(m), T(n_{\ell})) = \alpha]} + O\rb{\frac{\abs{\{\ell \leq L: (n_{\ell} \bmod k^{\lambda}) \notin S_{\lambda}\}}}{L}}.
  \end{align*}
  
  We find, since $n_{\ell}$ distributes regularly within residue classes and by~\eqref{eq_not_sparse},
  \begin{align*}
  	\lim_{\lambda\to \infty} \lim_{L \to \infty} \frac{\abs{\{\ell \leq L: (n_{\ell} \bmod k^{\lambda}) \notin S_{\lambda}\}}}{L} = \lim_{\lambda \to \infty} \sum_{r \notin S_{\lambda}} c(r; k^{\lambda}) = 0.
  \end{align*}
  
  Thus, we are interested in computing
  \begin{align}\label{eq: moyenne}
  	\sum_{m \in S_{\lambda}} \frac{1}{L} \sum_{\substack{\ell \leq L\\ n_{\ell} \equiv m \bmod k^{\lambda}}} \ind_{[f(s(m), T(n_{\ell})) = \alpha]} &= \sum_{q \in Q} \sum_{\substack{m \in S_{\lambda} \\ s(m) = q}} \frac{1}{L}  \sum_{\substack{\ell \leq L\\ n_{\ell} \equiv m \bmod k^{\lambda}}} \ind_{[f(q, T(n_{\ell})) = \alpha]}.
  \end{align}
  
  Hence, we are interested in finding an asymptotic distribution of $T(n_{\ell})$ for $n_{\ell} \equiv m \bmod k^{\lambda}$, which will be done via Lemma~\ref{le_distT}.

  We find by the Chinese Remainder Theorem and as $d$ and $k$ are coprime,
  \begin{align}\label{eq: equidistribution}
\nonumber  \lim_{L \to \infty} \frac{1}{L} \sum_{\substack{\ell \leq L}} D_j(T(n_{\ell})) &= 
	\lim_{L \to \infty} \frac{1}{L} \sum_{0\leq m <k^{\lambda}} \sum_{\substack{\ell \leq L\\ n_{\ell} \equiv m \bmod k^{\lambda}}} D_j(T(n_{\ell}))\\ 
		&=\sum_{0\leq m <k^{\lambda}} \sum_{0\leq i < d}  \lim_{L \to \infty} \frac{1}{L} \sum_{\substack{\ell \leq L\\n_{\ell} \equiv m \bmod k^{\lambda}\\ n_{\ell} \equiv i \bmod d}} \e\rb{\frac{j \cdot i}{d}}\\
\nonumber	& =  \sum_{0\leq m <k^{\lambda}}\sum_{0\leq i < d}  c(m;k^{\lambda}) c(i; d)  \e\rb{\frac{j \cdot i}{d}} \\	
  \nonumber	& =  \sum_{0\leq i < d} c(i; d) \cdot \e\rb{\frac{j \cdot i}{d}}.
  \end{align}
  This finishes the computation of the left-hand side of~\eqref{eq_dist} together with~\eqref{eq_D_not_D_j}.
  Thus, it remains to construct a suitable measure $\boldsymbol{\nu}$. 
  We define $\boldsymbol{\nu}$ as
  \begin{align*}
  	\boldsymbol{\nu}(g) := \frac{1}{\abs{G}}\sum_{0\leq j < d} \overline{D_j(g)} \sum_{0\leq i < d} c(i; d) \cdot \e\rb{\frac{i\cdot j}{d}}.
  \end{align*}
  If $g\in G_{s}$, we find
  \begin{align*}
  	\boldsymbol{\nu}(g) &= \frac{1}{\abs{G}} \sum_{0 \leq i < d} c(i;d) \sum_{0 \leq j < d} \e\rb{\frac{i \cdot j}{d}} \e\rb{\frac{-s\cdot j}{d}}\\
  		&= \frac{1}{\abs{G}} c(s;d) \cdot d.
  \end{align*}
  Thus, it follows directly that $\boldsymbol{\nu}$ is probability measure on $G$.
  Moreover, we compute
  \begin{align*}
  	\int_{G} D d\boldsymbol{\nu} &= \sum_{g\in G} D(g) \boldsymbol{\nu}(g)\\
  		&= \sum_{0\leq j < d} \sum_{0 \leq i < d} c(i;d) \e\rb{\frac{i \cdot j}{d}} \frac{1}{\abs{G}} \sum_{g \in G} D(g) \cdot \overline{D_j(g)}.
  \end{align*}
  As the irreducible and unitary representations are either equivalent or orthogonal, we find that $\int_{G} D \, d\boldsymbol{\nu} = 0$ unless $D = D_j$ for some $0\leq j < d$. Since 
  $$
  \frac{1}{\vert G\vert} \sum_{g\in G} D_j(g)\cdot \overline{D_j(g)} =1 \,,
  $$
  we find 
  $$\int_{G} D_j \, d\boldsymbol{\nu} = \sum_{0 \leq i < d} c(i;d) \e\rb{\frac{i \cdot j}{d}}\,.$$  
  By \eqref{eq: equidistribution}, this shows that the sequence $T(n_\ell)$ is $\boldsymbol{\nu}$-uniformly distributed in $G$. 
  Finally, we are able to simplify the expression for $\boldsymbol{\nu}$. We find for $g \in G_{i_0}$,
  \begin{align*}
  	\boldsymbol{\nu}(g) &= \frac{1}{\abs{G}} \sum_{0\leq j < d} \e\rb{\frac{-i_0 \cdot j}{d}} \sum_{0\leq i < d} c(i; d) \e\rb{\frac{i\cdot j}{d}}\\
  		&= \frac{1}{\abs{G}} \sum_{0\leq i < d} c(i; d) \sum_{0\leq j < d} \e\rb{\frac{j(i - i_0)}{d}}\\
  		&= \frac{1}{\abs{G}} \sum_{0\leq i < d} c(i; d) \cdot d \cdot \ind_{[i = i_0]}\\
  		&= \frac{d}{\abs{G}} \cdot c(i_0; d).
  \end{align*}
  
  Thus, we can apply Lemma~\ref{le_distT} and find for $g \in G_j$
  \begin{align*}
  	\lim_{L \to \infty} \frac{c(m;k^{\lambda})}{L}  \sum_{\substack{\ell \leq L\\ n_{\ell} \equiv m \bmod k^{\lambda}}} \ind_{[T(n_{\ell}) = g]} = \frac{d}{\abs{G}} \cdot c(j; d) = d(T(n_{\ell}), g).
  \end{align*}
 Coming back to \eqref{eq: moyenne},  this shows in total
  \begin{align*}
  	\lim_{L \to \infty} \frac{1}{L} \sum_{\ell \leq L} \inda &= \sum_{m \in S_{\lambda}} c(m; k^{\lambda}) \sum_{g\in G} d(T(n_{\ell}), g) \ind_{[f(s(m), g) = \alpha]} + O\rb{\sum_{m \notin S_{\lambda}} c(m, k^{\lambda})}\\
  		&= \sum_{g\in G} d(T(n_{\ell}), g) \sum_{q \in Q} \sum_{\substack{m \in S_{\lambda}\\ s(m) = q}} c(m; k^{\lambda}) \ind_{[f(q, g) = \alpha]} + O\rb{\sum_{m \notin S_{\lambda}} c(m, k^{\lambda})}.
  \end{align*}
  The result follows now for $\lambda \to \infty$.
  \end{proof}

  	It turns out that~\eqref{eq_D_not_D_j} is usually the most challenging part about applying Theorem~\ref{th_density_primitive}.

  \section{The subsequence along primes}\label{sec:primes}
  
  We apply in this section Theorem~\ref{th_density_primitive} to the subsequence along primes 
which reproves results from~\cite{Muellner2017} using this new framework.
For this purpose we are repeating the key arguments from~\cite{Muellner2017}.
  
 We find directly by the Prime Number Theorem in arithmetic progressions that
 \begin{align*}
 	\lim_{N \to \infty} \frac{1}{\pi(N)} \sum_{p\leq N} \ind_{[p \equiv r \bmod m]} =  \frac{\ind_{[(r,m) = 1]}}{\varphi(m)} = c(r; m).
 \end{align*}
 One finds directly that primes distribute regularly within residue classes.
 Moreover, $(r, k^{\lambda}) = 1 \Leftrightarrow (r, k) = 1$ holds with positive probability.
 Thus, $c(r;k^{\lambda})$ resembles a uniform distribution on a subset of $[0, \ldots, k^{\lambda}-1]$ with positive density (independent of $\lambda$).
 This shows \eqref{eq_not_sparse} as $\lim_{\lambda \to \infty} \frac{\abs{S_{\lambda}}}{k^{\lambda}} = 0$.
 
  Thus, it remains to show for any $m, h \in \N$,
  \begin{align}\label{eq_saving_primes}
  	\lim_{N \to \infty} \frac{1}{\pi(N)} \norm{\sum_{\substack{p \leq N\\ p \equiv m \bmod h}} D(T(p))} = 0.
  \end{align} 
  The key ingredient was to generalize and use a method developed by Mauduit and Rivat~\cite{Mauduit2015}.
  We will focus here mainly on the generalized version, as it proved to be better applicable in this situation.
  We fix some $k \in \N$ and let $f_{\lambda}(n)$ denote $f(n \bmod k^{\lambda})$ and let $.^{H}$ denote the Hermitian transpose.
  We also need the following two definitions.
  \begin{definition}\label{def:1}
  A function $f:\N\rightarrow \U_d$ has the \emph{Carry property} if there exists $\eta >0$ such that uniformly for $(\lambda,\alpha, \rho)\in \N^3$
  with $\rho<\lambda$, the number of integers $0\leq \ell <k^{\lambda}$ such that there exists $(n_1,n_2) \in \{0,\ldots,k^{\alpha}-1\}^2$ with
  \begin{align}\label{eq:carry_violation}
    f(\ell k^{\alpha} + n_1 + n_2)^{H} f(\ell k^{\alpha} + n_1) \neq f_{\alpha + \rho}(\ell k^{\alpha} + n_1 + n_2)^{H} f_{\alpha + \rho}(\ell k^{\alpha} + n_1)
  \end{align}
  is at most $O(k^{\lambda-\eta \rho})$ where the implied constant may depend only on $k$ and $f$.
\end{definition}
\begin{definition}\label{def:2}
  Given a non-decreasing function $\gamma: \R \rightarrow \R$ satisfying $\lim_{\lambda \rightarrow \infty} \gamma(\lambda) = +\infty$ and $c>0$ we let 
  $F_{\gamma, c}$ denote the set of functions $f:\N\rightarrow \U_d$ such that for $(\alpha,\lambda)\in \N^2$ with $\alpha \leq c \lambda$ and $t\in \R$:
  \begin{align}
    \norm{k^{-\lambda} \sum_{u<k^{\lambda}} f(u k^{\alpha}) \e(-ut)}_F \leq k^{-\gamma(\lambda)}.
  \end{align}
  We say in this case that $f$ has the \emph{Fourier property}.
\end{definition}

The main difference between the given definitions and the ones used by Mauduit and Rivat is that they only considered complex-valued functions $f$ and a stronger Carry Property corresponding to $\eta = 1$.
The proof of the following result (to be found in~\cite{Muellner2017}) is in very large parts identical to the corresponding proof in~\cite{Mauduit2015}.

\begin{theorem}[Theorem 4.3 in \cite{Muellner2017}]\label{thm:prime}
  Let $\gamma:\R\to\R$ be a non-decreasing function satisfying $\lim_{\lambda\to\infty}\gamma(\lambda) = + \infty$, and $f:\N\to \U_d$ be a function satisfying
  Definition~\ref{def:1} for some $\eta \in (0,1]$ and $f\in \mathcal{F}_{\gamma,c}$ for some $c\geq 10$ in Definition~\ref{def:2}.
  Then for any $\theta \in \R$ we have
  \begin{align}
    \norm{\sum_{n\leq x} \Lambda(n)f(n)\e(\theta n)} \ll c_1(k) (\log x)^{c_2(k)} x k^{-\eta \gamma(2\floor{(\log x)/(80 \log k)})/20},
  \end{align}
  with the same constants as in \cite{Mauduit2015}\footnote{This estimate is non-trivial whenever $\log(x) = o(\gamma(x))$.}.
\end{theorem}
  The factor $\e(\theta n)$ is in particular useful to detect $n \equiv a \bmod m$ via the following identity,
  \begin{align*}
  	\frac{1}{m} \sum_{0 \leq h < m} \e\rb{\frac{h(n-a)}{m}} = \ind_{[n \equiv a \bmod m]}.
  \end{align*}
  
  Furthermore, it is classical to replace estimates for the sum along primes by correlations with $\Lambda$.
  This gives (for example by~\cite{Iwaniec2004})
  \begin{align*}
  \norm{\frac{1}{\pi(N)} \sum_{p<N} f(p) \e(p\theta)}_F &\ll \frac{1}{\log N} \frac{1}{\pi(N)} \max_{t\leq N} \norm{\sum_{n\leq t} \Lambda(n) f(n) \e(n\theta)}_{F} + O(\sqrt{N}).
\end{align*}
  The Prime Number Theorem ensures that $\log(N) \cdot \pi(N) \sim N$.
  In total we find
  \begin{align*}
  	\norm{\frac{1}{\pi(N)} \sum_{\substack{p< N\\ p \equiv a \bmod m}} f(p)} &\leq \max_{\theta} \norm{\frac{1}{\pi(N)} \sum_{p< N} f(p) \e(p \theta)}\\
  		&\ll \frac{1}{N} \max_{t\leq N} \max_{\theta}\norm{\sum_{n\leq t} \Lambda(n) f(n) \e(n\theta)}_{F} + O(\sqrt{N}).
  \end{align*}
  Thus we have shown the following corollary.
  \begin{corollary}
  	Let $\gamma:\R\to\R$ be a non-decreasing function satisfying $\lim_{\lambda\to\infty}\gamma(\lambda)/\log(\lambda) = + \infty$, and $f:\N\to \U_d$ be a function satisfying
  Definition~\ref{def:1} for some $\eta \in (0,1]$ and $f\in \mathcal{F}_{\gamma,c}$ for some $c\geq 10$ in Definition~\ref{def:2}.
  Then for any $a, m \in \N$ we have
  \begin{align*}
  	\lim_{N\to \infty} \norm{\frac{1}{\pi(N)} \sum_{\substack{p< N\\ p \equiv a \bmod m}} f(p)} = 0.
  \end{align*}
  \end{corollary}
  
  Naturally, we are now interested in the case $f(n) := D(T(n))$.
    The Carry-Property for $D(T(n))$ follows from the way that the digital representation of $n$ influences $T(n)$.  
  \begin{proposition}[\cite{Muellner2017}]
  	Let $D$ be any unitary and irreducible representation of $G$. Then $D(T(.))$ satisfies Definition~\ref{def:1} for some $\eta > 0$.
  \end{proposition}
  The much more challenging part was to show the Fourier Property.
  \begin{proposition}[\cite{Muellner2017}]\label{pr_fourier}
	  Let $D$ be a unitary, irreducible representation of $G$ different from $D_{j}$. 
	  Then $D(T(.)) \in F_{\gamma, c}$ for some linear $\gamma$ and all $c \geq 0$.
  \end{proposition}
  We note that Proposition~\ref{pr_fourier} does not hold for $D = D_{j}$, in fact for $t = j/d$ we have
  \begin{align*}
  	\frac{1}{N} \sum_{n < N} D_{j}(n) \e(-n\ell/d) = \frac{1}{N} \sum_{n<N} \e(0) = 1.
  \end{align*}
    
    Thus we have shown~\eqref{eq_saving_primes} and are able to apply Theorem~\ref{th_density_primitive} which gives the following result.
    \begin{proposition}
    	Let $a(n)$ be a prolongable and primitive automatic sequence. Then the density of $a(n) = \alpha$ exist along the subsequence of primes.
    \end{proposition}

  \medskip

 Finally we prove Theorem~\ref{th_primes_mod_m} saying that there exists $m$ with
 $d_{\log}(a(p_n), \alpha) = d_{\log}(a(n_\ell),\alpha)$, where $n_\ell$ runs through
 all positive integers with $(n,m) = 1$.

	\begin{proof}
		We first use Theorem~\ref{Th1} to find $\mathcal{B} = \{b_1(n),\ldots, b_s(n)\}$.
		Each of the $b_i$ can be written as $b_i(n) = f_i(s_i(n), T_i(n))$, with some $d_i = d(b_i)$.
		We choose now $m = k \cdot \prod d_i$ and let $n_{\ell}$ denote the sequence of integers that are coprime to $m$.
		
		We find immediately that $n_{\ell}$ distributes regularly within residue classes and that $c_{n_{\ell}}(m; h)$ fulfills~\eqref{eq_not_sparse}.
		Furthermore,~\eqref{eq_D_not_D_j} is an immediate consequence of Proposition~\ref{pr_fourier}.
		Thus, we can apply for any $b_i$ Theorem~\ref{th_density_primitive} both for the subsequence along $\P$ and along $n_{\ell}$.
		
		A simple computation shows that $$c_{\P}(r; k^{\lambda}) = c_{n_{\ell}}(r; k^{\lambda}) = \ind_{(r,k) = 1} \frac{1}{\varphi(k) k^{\lambda-1}}$$ and $$c_{\P}(r; d_i) = c_{n_{\ell}}(r; d_i) = \ind_{(r, d_i) = 1} \frac{1}{\varphi(d_i)}.$$
		This shows immediately that $d(b_i(p_n), \alpha) = d(b_i(n_{\ell}), \alpha)$ for all $1\leq i \leq s$.
		The result follows now directly from Equation~\eqref{eq_d_log}.
	\end{proof}

We remark that Theorem~\ref{th_primes_mod_m} can be also used to observe zero densities.
Namely, we have $d_{log}(a(mn + r), \alpha) = 0$ if and only if  $d_{log}(a(p_n), \alpha) = 0$
for all $r$ with $(r,m) = 1$.

  \section{The subsequence along squares}\label{sec:squares}
  
	The goal of this section is to compute the density of primitive automatic sequences along squares.
	There are already some interesting results in this direction that we want to mention here.
	The first and ground-breaking result is due to Mauduit and Rivat~\cite{Mauduit2009}, where they showed that the Thue-Morse sequence takes values $0$ and $1$ with density $\frac{1}{2}$ along squares. This result relies on $L^{1}$ estimates of the Fourier-Transform and is thus not possible to extend to general automatic sequences. 
	However, it was generalized to invertible automatic sequences by Drmota and Morgenbesser~\cite{Drmota2012}.
	Moreover, there are results about the density of blocks along squares (i.e. normality) for the Thue-Morse sequence by Drmota, Mauduit and Rivat~\cite{Drmota2019} and, more generally, strongly block-additive functions mod $m$ by the last author~\cite{Muellner2018}.
	
	Finally, and most important for this section, there is a new result by Mauduit and Rivat~\cite{Mauduit2018} which gives density results along squares, for all functions satisfying the Carry-Property and the Fourier-Property (again in the stricter sense). In particular, they only consider complex-valued sequences $f$, and a stronger Carry-Property, i.e. $\eta = 1$. 
	
	The main result of this section is the following theorem.
	\begin{theorem}\label{th_squares}
		Let $a(n)$ be a prolongable and primitive automatic sequence.
		With the notation from~\eqref{eq_structure_a}, we write $a(n) = f(s(n), T(n))$.
		Then there exist the densities $d_{q} = d(s(n^2), q)$ and $d_g = d(T(n^2), g)$.
		Furthermore, we have
		\begin{align*}
			d(a(n^2), \alpha) = \sum_{q \in Q, g \in G} d_{q} \cdot d_g \cdot \ind_{[f(q, g) = \alpha]}.
		\end{align*}
	\end{theorem}
	
	Naturally, the idea is to apply Theorem~\ref{th_density_primitive} for the subsequence $n_{\ell} = \ell^2$.
	Thus, the proof splits into two parts. We first aim to show~\eqref{eq_not_sparse} and then~\eqref{eq_D_not_D_j}.
	
	\subsection{Synchronizing automatic sequence along squares}
	The main result of this subsection is the following proposition.
	\begin{proposition}\label{prop: squares}
		The subsequence along squares distributes regularly within residue classes and fulfill~\eqref{eq_not_sparse}.
	\end{proposition}

We first observe that 
		\begin{align*}
			c(m; h) := \frac{\{0 \leq x < h: x^2 \equiv m \bmod h\}}{h} \,.
		\end{align*}
		This already shows that the subsequence along squares distributes regularly within residue classes by the Chinese Remainder Theorem. It thus remains to prove \eqref{eq_not_sparse}. 
		
As $c(m; h)$ is multiplicative in the second coordinate, we are interested in $c(m; p^{\alpha})$, where $p$ is a prime.
	We will use the following results which follow directly from Hensel's Lemma.
	\begin{lemma}\label{le_cor_hensel}
		Let $p$ be an odd prime and $\alpha \geq 1$.
		Then we have for $m \not \equiv 0 \bmod p^{\alpha}$ and any $\ell \geq 0$,
		\begin{align*}
			c(m; p^{\alpha +\ell}) = \frac{c(m; p^{\alpha})}{p^{\ell}}.
		\end{align*}
		Furthermore, if $\alpha \geq 3, m \not \equiv 0 \bmod p^{\alpha-2}$, then for any $\ell \geq 0$,
		\begin{align*}
			c(m; 2^{\alpha +\ell}) = \frac{c(m; 2^{\alpha})}{2^{\ell}}.
		\end{align*}
	\end{lemma}
	\begin{corollary}\label{cor_lift}
		Let $k = p_1^{\alpha_1} \cdot \ldots \cdot p_s^{\alpha_s}$, where $p_i \in \P$.
		Let $\lambda \in \N$ and $a$ such that for all $i$,
		\begin{align*}
			m \not \equiv 0 \bmod p_i^{\lambda \alpha_i - 2}.
		\end{align*}
		Then, $c(m; k^{\lambda + \ell}) = \frac{c(m; k^{\lambda})}{k^{\ell}}$ for all $\ell \geq 0$.
	\end{corollary}

We are now ready to prove Proposition~\ref{prop: squares}. 
	
	\begin{proof}[Proof of Proposition \ref{prop: squares}]This will allow us to show the following result.
		\begin{align}\label{eq_most_qr}
			\lim_{\lambda \to \infty} \sum_{\substack{m<k^{\lambda}\\\forall i: m \not \equiv 0 \bmod p_i^{\lambda \alpha_i -2} }} c(m; k^{\lambda}) = 1.
		\end{align}
		By the Chinese Remainder Theorem it is sufficient to show
		\begin{align*}
			\lim_{\lambda \to \infty} \sum_{\substack{m<p_i^{\alpha_i \lambda}\\m \not \equiv 0 \bmod p_i^{\lambda \alpha_i -2}}} c(m; p_i^{\alpha_i \lambda}) = 1
		\end{align*}
		or equivalently
		\begin{align*}
			\lim_{\lambda \to \infty} \sum_{\substack{m<p_i^{2}}} c(m p_i^{\alpha_i \lambda -2}; p_i^{\alpha_i \lambda}) = 0.
		\end{align*}
		We conclude the proof of~\eqref{eq_most_qr} by noting
		\begin{align*}
			&\abs{\{x < p_i^{\alpha_i \lambda}: x^2 \equiv m p_i^{\alpha_i \lambda-2} \bmod p_i^{\alpha_i \lambda}\}}\\
			 &\qquad \qquad \leq \abs{\{x < p_i^{\alpha_i \lambda}: x \equiv 0 \bmod p_i^{\floor{\alpha_i \lambda/2 -1}}\}} 
			 = p_i^{\ceil{\alpha_i \lambda/2} +1} = o_{\lambda \to \infty}(p_i^{\alpha_i \lambda}).
		\end{align*}
		
		We find by Corollary~\ref{cor_lift}
		\begin{align*}
			\sum_{m \in S_{\lambda_1 + \lambda_2}} c(m; k^{\lambda_1 + \lambda_2}) &\geq  \sum_{\substack{m_1 < k^{\lambda_1}\\ \forall i: m_1 \not \equiv 0 \bmod p_i^{\lambda_1 \alpha_i-2}}} \, \sum_{\substack{m_2 < k^{\lambda_2}\\ m_2 k^{\lambda_1} + m_1 \in S_{\lambda_1 + \lambda_2}}} c(m_2 k^{\lambda_1} + m_1; k^{\lambda_1 + \lambda_2})\\
				&= \sum_{\substack{m_1 < k^{\lambda_1}\\ \forall i: m_1 \not \equiv 0 \bmod p_i^{\lambda_1 \alpha_i-2}}} c(m_1; k^{\lambda_1 + \lambda_2}) \cdot \frac{1}{k^{\lambda_2}} \sum_{\substack{m_2 < k^{\lambda_2}\\ m_2 k^{\lambda_1} + m_1 \in S_{\lambda_1 + \lambda_2}}} 1.
		\end{align*}
		We recall that any concatenation of a synchronizing word with any other word is again synchronizing.
		Thus, $m_2 \in S_{\lambda_2}$ implies $m_2 k^{\lambda_1} + m_1 \in S_{\lambda_1 + \lambda_2}$ giving in total
		\begin{align*}
			\sum_{m \in S_{\lambda_1 + \lambda_2}} c(m; k^{\lambda_1 + \lambda_2}) \geq \rb{\sum_{\substack{m_1 < k^{\lambda_1}\\ \forall i: m_1 \not \equiv 0 \bmod p_i^{\lambda_1 \alpha_i-2}}} c(m_1; k^{\lambda_1})} \cdot \rb{\frac{1}{k^{\lambda_2}}\sum_{m_2  \in S_{\lambda_2}} 1},
		\end{align*}
		which finishes the proof as both factors tend to $1$ for $\lambda_1, \lambda_2 \to \infty$.
	\end{proof}

	\subsection{Generalizing the result of Mauduit and Rivat for squares}
	The main result of this section is the following theorem.
	\begin{theorem}\label{th_squares_general}
		Let $\gamma: \R \to \R$ be a nondecreasing function satisfying $\lim_{\lambda\to\infty}\gamma(\lambda) = \infty$, and let $f: \N \to \U_d$ be a function satisfying Definition~\ref{def:1} for some $\eta > 0$ and $f \in F_{\gamma, c}$ for some $c \geq 18$ in Definition~\ref{def:2}. Then for any $\theta \in \R$, we have
		\begin{align*}
			\norm{\sum_{0<n \leq x} f(n^2) \e(n \theta)}_2 \ll_{d,f,k} (\log x)^{\omega(q) + 2} \rb{x k^{-\frac{\eta \gamma(2 \floor{(3 \log x)/(100 \log k)})}{56}}},
		\end{align*}
		where the absolute implied constant only depends on $d,f$ and $k$.
	\end{theorem}
	The proof works exactly the same as in~\cite{Mauduit2018}, one only needs to account for the newly introduced constant $\eta>0$ and needs to be more careful as matrices do not commute. This part should be read alongside with~\cite{Mauduit2018} and we only comment on the necessary changes briefly.

	Lemma 1 of~\cite{Mauduit2018} can be easily adapted to matrix valued sequences:
	\begin{lemma}\label{le_vdc}
		For all $z_1,\ldots,z_N$ being complex $d\times d$ matrices and all integers $k \geq 1$ and $R \geq 1$, we have
		\begin{align*}
			\norm{\sum_{1\leq n \leq N} z_n}^{2}_{F} &\leq \frac{N + kR - k}{R} \rb{\sum_{1\leq n \leq N} \norm{z_n}^{2}_{F} + 2 \sum_{1\leq r < R}\rb{1-\frac{r}{R}} \sum_{1 \leq n \leq N - kr} \tr(z_{n+kr} z_n^{H})}\\
				&= \frac{N + kR - k}{R} \rb{\sum_{1\leq n \leq N} \norm{z_n}^{2}_{F} + 2 \sum_{1\leq r < R}\rb{1-\frac{r}{R}} \sum_{1 \leq n \leq N - kr} \tr(z_{n} z_{n+kr}^{H})}
		\end{align*}
	\end{lemma}

	The next few lemmata can stay completely unchanged. The next one that needs to be changed is Lemma 7:
	\begin{lemma}
		Let $f: \N \to \U_d$ satisfying Definition~\ref{def:1}, and let $(\nu, \kappa, \rho) \in \N^3$ with $3\rho <\nu < \kappa < \nu + 2 \rho$. The set $\mathcal{E}$ of $n \in \{k^{\nu-1},\ldots,k^{\nu}-1\}$ such that there exists $\ell \in \{0,\ldots,k^{\kappa}-1\}$ with $f(n^2+k)^{H} f(n^2) \neq f_{\kappa + \rho}(n^2 + k)^{H} f_{\kappa + \rho}(n^2)$ satisfies
		\begin{align*}
			\card \mathcal{E} \ll \card_{f,k} k^{\nu-\frac{\eta \rho}{2}}.
		\end{align*}
	\end{lemma}
	The proof stays essentially unchanged, but it will be important later that one takes the hermitian of $f(n^2+\ell)$ instead of $f(n^2)$.
	
	For Lemma 8 we only need to change the final conclusion to
	\begin{align*}
		\card \mathcal{E} \ll_{f,k} k^{\nu - \eta(\nu_1-\nu_0)} + k^{\frac{\nu_2}{2} + \nu_2 - \nu_0} \log(k^{\nu_2}) k^{-\eta(\nu_1 - \nu_0)}.
	\end{align*}

	We are now ready to tackle the estimate of
	\begin{align*}
		S_0 := \sum_{N/2 < n \leq N} f(n^2) \e(\theta n).
	\end{align*}
	We apply Lemma~\ref{le_vdc} to cut off high digits, just as in ~\cite{Mauduit2018}. We only need to replace,
	\begin{align*}
		S_1(r) = \sum_{n \in I_1(N,r)} f((n+r)^2)^{H} f(n^2) \e(\theta r).
	\end{align*}
	Next we apply Lemma~\ref{le_vdc} again to cut off low digits. The resulting sum is
	\begin{align*}
		S_2'(r,s) = \sum_{n \in I_2(N,r,s)} f_{\nu_2}((n+r+sk^{\nu_1})^2)^{H} f_{\nu_2}((n + sk^{\nu_1})^2) f_{\nu_2}(n^2)^{H} f_{\nu_2}((n+r)^2).
	\end{align*}
	However, as we are only interested in the trace of $S_2'(r,s)$, we are able to replace it with
	\begin{align*}
		\sum_{n \in I_2(N,r,s)} f_{\nu_2}((n + sk^{\nu_1})^2) f_{\nu_2}(n^2)^{H} f_{\nu_2}((n+r)^2) f_{\nu_2}((n+r+sk^{\nu_1})^2)^{H}.
	\end{align*}
	After having this explicit order for the terms, we can use the definition of the double truncated version, to rewrite it as
	\begin{align*}
		\sum_{n \in I_2(N,r,s)} f_{\nu_1,\nu_2}((n + sk^{\nu_1})^2) f_{\nu_1,\nu_2}(n^2)^{H} f_{\nu_1,\nu_2}((n+r)^2) f_{\nu_1,\nu_2}((n+r+sk^{\nu_1})^2)^{H}.
	\end{align*}
	Thereafter, most things stay unchanged, one only needs to be careful to not change the order of the terms, e.g. the estimate for $S_4'(r,s)$ becomes
	\begin{align*}
		S_4'(r,s) \leq k^{2(\nu_2-\nu_0)} &\sum_{\abs{h_1} \leq H} \abs{a_{h_1}(k^{\nu_0-\nu_2},H)}^{2}\\
			& \sum_{0\leq h_2 < k^{\nu_2-\nu_0}} \sum_{0 \leq h_3 < k^{\nu_2-\nu_0}} \abs{\hat{g}(-h_2) \hat{g}(-h_2-h_1) \hat{g}(h_3-h_1) \hat{g}(h_3)}\\
			&\abs{\sum_{n \in I_2(N,r,s)} \e\rb{\frac{2 h_1 r + 2 (h_2 + h_3) s k^{\nu_1}}{k^{\nu_2}} \cdot n}}.
	\end{align*}
	However, we have now set up the order of the terms, such that everything works out just as in~\cite{Mauduit2018}.
	
	The next more substantial change has to be made to Lemma 9, where the final estimate needs to be replaced by
	\begin{align*}
		\sum_{0\leq h < k^{\nu_2 - \nu_0}} \sum_{0\leq \ell < k^{\nu_2 - \nu_0 - \lambda}} \norm{\hat{g}(h + \ell) \hat{g}(h)}_2^2 \ll_{d,f,k} k^{\eta \frac{\nu_1 - \nu_0 - \gamma(\lambda)}{2}} (\log(k^{\nu_2 - \nu_1}))^2,
	\end{align*}
	i.e. we needed to take $\eta$ into account, which was equal $1$ in~\cite{Mauduit2018}.
	This leads to the estimate
	\begin{align*}
		\frac{1}{R} \sum_{1\leq r < R} S_8(r) \ll k^{\nu + \eta \frac{\nu_1 - \nu_0 - \gamma(\nu_2 - \nu_0 - 2 \rho)}{2}} (\log k^{\nu_2 - \nu_1})^2 + \rho k^{\nu - \rho} \log k.
	\end{align*}
	The rest of the proof stays unchanged and one only needs to choose the values for $\rho'$ differently, i.e.
	\begin{align*}
		\rho' = \floor{\eta \frac{\gamma(2 \rho)}{7}}.
	\end{align*}
	This finishes the proof of Theorem~\ref{th_squares_general}.

	\subsection{Finishing the proof of Theorem~\ref{th_squares} and Theorem~\ref{th_squares_log}}
	We have already seen that $n^2$ distributes regularly within residue classes and that $c_{n^2}(m; h)$ satisfies~\eqref{eq_not_sparse}.
	It remains to apply Theorem~\ref{th_squares_general} to the function $f(n) = D(T(n))$ for unitary and irreducible representations $D$ different from $D_j$. Again the factor $\e(n \theta)$ can be used to detect the residue of $n^2$ modulo $k^{\lambda}$.
	Thus, we can apply Theorem~\ref{th_density_primitive} to the subsequence along squares, which gives immediately Theorem~\ref{th_squares}.
	
	The main part of Theorem~\ref{th_squares_log} is now an immediate consequence of Theorem~\ref{Th1}. 
	It will just remain to prove that when the input base $k$ is prime, then the densities are computable rational numbers.

\section{Computability of densities along subsequences}\label{sec:computability}

We first start with the primitive and prolongable case.
\subsection{Densities of primitive automatic sequences}

We use this section to recall a classical results about densities of pure, primitive and prolongable $k$-automatic sequences.
Therefore, we need the following definition.
\begin{definition}\label{def_incidence}
	Let $A = (Q, \{0,\ldots,k-1\}, \delta, q_0)$ be a DFA, where $Q = \{q_0,q_1, \ldots,q_d\}$.
	We define the \emph{incidence matrix $M = M(A)$} as follows:
	\begin{align*}
		M = (m_{i,j})_{0\leq i,j \leq d},
	\end{align*}
	where $m_{i,j} = \abs{\{0\leq w < k: \delta(q_j, w) = q_i\}}$.
\end{definition}
One sees directly that $\sum_{0\leq i \leq d} m_{i,j} = k$ for all $0\leq j \leq d$. Thus one has that $(1,1, \ldots,1)$ is a left-eigenvector assoziated with the eigen-value $k$. It turns out that the right-eigenvector associated to the eigenvalue $k$ describes the densities.

\begin{theorem}[Theorem 8.4.7 and 8.4.5 of~\cite{Allouche2003}]\label{th_primitive_density}
	Let $(a(n))_{n\geq 0}$ be a pure and primitive $k$-automatic sequence with incidence matrix $M$, as in Definition~\ref{def_incidence}. Moreover, let $v = (v_0, \ldots, v_d)^{T}$ be the positive normalized right-eigenvector of $M$ associated with the eigenvalue $k$.
	Then
	\begin{align*}
		d(a(n), q_i) = v_i \in \Q_{> 0},
	\end{align*}
	for all $0\leq i \leq d$.
\end{theorem}

\begin{example}\label{ex_paper_folding}
	We discuss the paperfolding sequence with respect to Theorem~\ref{th_primitive_density}.
		The transition diagram of the paperfolding sequence is given below.

		
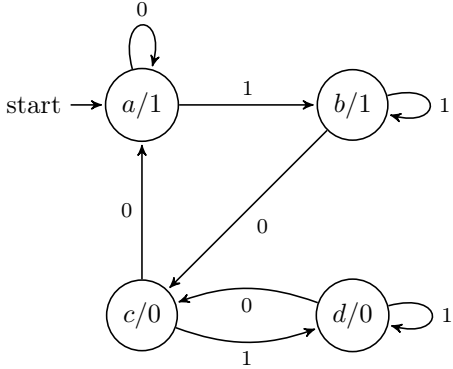
\begin{tikzpicture}[->,>=stealth',shorten >=1pt,auto,node distance=2.8cm, semithick, bend angle = 20, cross 		line/.style={preaction={draw=white,-,line width=4pt}}]
    
    		\node[state, initial](A)                    {$a/1$};
    		\node[state]         (B) [right of=A] 			{$b/1$};
    		\node[state]	(C) [below of =A]		{$c/0$};
    		\node[state]	(D) [below of =B]		{$d/0$};
    
    		\path [every node/.style={font=\footnotesize}, pos = 0.5]
    		(A) edge 		  node 		 {1}      (B)
			edge [loop above] 	node		{0} 	(A)
		(B) edge 	[loop right]		node 		{1} 	(B)
    			edge			node		{0}	(C)
    		(C) edge 			node		{0} 	(A)
			edge	[bend right]	node	[below]	{1}	(D)
    		(D) edge [loop right] 	node		{1}	(D)
			edge [bend right]	node		{0}	(C);
	\end{tikzpicture}

	Thus, we find that the transition matrix is given by
	\begin{align*}
		M = \begin{pmatrix}
1 & 0 & 1 & 0 \\
1 & 1 & 0 & 0 \\
0 & 1 & 0 & 1  \\
0 & 0 & 1 & 1 
\end{pmatrix},
	\end{align*}
with the unique normalized eigenvector $(1/4, 1/4, 1/4, 1/4)^T$ assoziated with the eigenvalue $2$ and consequently, both the value $0$ and $1$ have density $1/2$.
\end{example}

\subsection{Primitive automatic sequences along primes}
We recall here how to explicitly compute the densities of primitive automatic sequences along primes.
We only consider the case when $a$ is pure as the general case follows immediately.
Therefore, let $(a(n))_{n\geq 0}$ be a primitive and prolongable $k$-automatic sequence.
Next we consider the (explicitly computable) decomposition in~\eqref{eq_structure_a}, i.e. $a(n) = f(s(n), T(n))$, where $s(n)$ is a pure synchronizing automatic sequence and $T(n)$ takes values in a finite group $G$. 
Then we computed $d = d(a)$\footnote{This is again explicit, as there are only finitely many possibilities to check.}.

Thus, we know by Theorem~\ref{th_primes_mod_m} that
\begin{align*}
	d(a(p_n), \alpha) = \frac{1}{\varphi(m)}\sum_{\substack{r<m, (r,m) = 1}} d(a(mn +r), \alpha),
\end{align*}
where $m = k \cdot d$.
Moreover, the sequence $((a(nm), \ldots, a(nm+m-1)))_{n\geq 0}$ is again a primitive and prolongable $k$-automatic sequence which is usually called the $m$-compression.
It can be for example constructed by starting with the state $(a(0), \ldots, a(m-1))$. Then we define $\delta((q_{i_0}, \ldots, q_{i_{m-1}}), j)$ as the set of $m$ states that we obtain by first writing the word of length $mk$,
\begin{align*}
	\delta(q_{i_0}, 0), \delta(q_{i_0}, 1), \ldots, \delta(q_{i_0}, k-1), \delta(q_{i_1}, 0), \ldots, \delta(q_{i_{m-1}}, k-1)
\end{align*}
and then picking the subword at positions $jm, \ldots jm + m-1$.\footnote{This construction is much more intuitive when working with substitutions instead of automata.}
It just remains to add new states until the automata is closed under this new transition function.
Then $(a(nm +r))_{n\geq 0}$ is the projection of this newly defined automatic sequence onto the $r$-th coordinate.

We finally note that $d(a(mn+r), \alpha) \in \Q$ and, therefore, $d(a(p_n), \alpha) \in \Q$.

\begin{example}
	We continue the discussion of the paper-folding sequence from Example~\ref{ex_paper_folding}.
	We see directly, that the paper-folding sequence is synchronizing. Thus, $T(n) = id$ and $G = \{id\}$ are trivial and $m = k = 2$ as $d=1$.

	Thus, we need to consider the $2$-compression of $a(n)$. The corresponding transition diagram is given below,
	
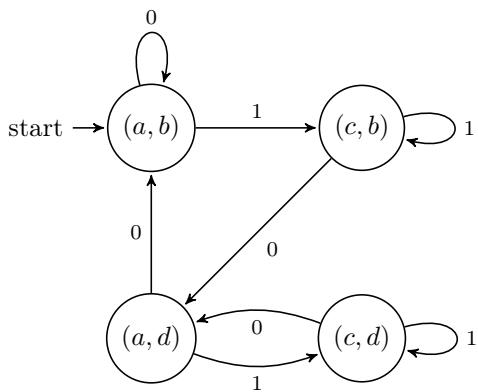
\begin{tikzpicture}[->,>=stealth',shorten >=1pt,auto,node distance=2.8cm, semithick, bend angle = 20, cross 		line/.style={preaction={draw=white,-,line width=4pt}}]
    
    		\node[state, initial](A)                    {$(a,b)$};
    		\node[state]         (B) [right of=A] 			{$(c,b)$};
    		\node[state]	(C) [below of =A]		{$(a,d)$};
    		\node[state]	(D) [below of =B]		{$(c,d)$};
    
    		\path [every node/.style={font=\footnotesize}, pos = 0.5]
    		(A) edge 		  node 		 {1}      (B)
			edge [loop above] 	node		{0} 	(A)
		(B) edge 	[loop right]		node 		{1} 	(B)
    			edge			node		{0}	(C)
    		(C) edge 			node		{0} 	(A)
			edge	[bend right]	node	[below]	{1}	(D)
    		(D) edge [loop right] 	node		{1}	(D)
			edge [bend right]	node		{0}	(C);
	\end{tikzpicture}

	We note that this is basically the same transition diagram as for the original paper-folding sequence. Thus, the density of every state is again $1/4$.
	However, now we need to consider the projection to the first coordinate which shows that the density of $b$ and $d$ are $1/2$ and the density of $a$ and $c$ are $0$. Thus we conclude that the density of the symbols $0$ and $1$ in the subsequence of the paperfolding sequence along the primes are $1/2$. 
\end{example}

	\subsection{Primitive automatic sequences along squares}
	For the sake of simplicity we only consider the case, where the
base $k$ is prime. The general case is much more technical, but the densities 
can be computed explicitly by Theorem~\ref{th_density_primitive} even if it is not clear whether they will be rational.
	\begin{theorem}\label{th_squares_primitive}
		Let $k$ be a power of a prime number and $a(n)$ a primitive and prolongable $k$-automatic sequence.
		Then the density along squares is rational.
	\end{theorem}
	\begin{proof}
		
		As $d(T(n^2), g) = \frac{d}{\abs{G}} \cdot c(j;d) \in \Q$ for $g \in G_j$, we see that we only need to consider the synchronizing part, i.e. we need to show that
		\begin{align*}
			\lim_{\lambda \to \infty} \sum_{m < k^{\lambda}} \ind_{[s(m) = q]} c(m; k^{\lambda}) \in \Q.
		\end{align*}
		Since $s(n)$ is a synchronizing $k$-automatic sequence, for $k = p^{\alpha}$, we know by~\cite[Proposition 3.3]{Byszewski2020} that it is also $p$-automatic and synchronizing. Thus, we assume without loss of generality that $k$ is a prime from now on.
		We first consider the case when $k=2$.
		We already know that the limit above exists, so we pass to the subsequence $2\lambda + 1$ to determine it.
		First we note that we can ignore $m = 0$ as $c(0; k^{2\lambda+1}) = k^{-\lambda-1} \to 0$. Then we rewrite $m = m'k^{2\mu+1} + m'_0 k^{2\mu}$ for some $m'_0 \neq 0$ and $0\leq \mu \leq \lambda$. Since $k$ is assumed to be prime, we have by Lemma~\ref{le_cor_hensel} that $$c(m; k^{2\lambda+1}) = \frac{c(m'_0 k^{2\mu}; k^{2 \mu +1})}{k^{2\lambda-2\mu}}.$$
		We can also determine $c(m'_0 k^{2\mu}; k^{2\mu +1})$ quite easily, as $x^2 \equiv m'_0 k^{2\mu} \bmod k^{2\mu+1}$ if and only if $x = x' k^{\mu+1} + x'_0k^{\mu}$ where $(x'_0)^2 \equiv m'_0 \bmod k$. Thus we have that $c(m'_0 k^{2\mu}; k^{2\mu +1}) = 2/k^{\mu+1}$ if $m'_0$ is a quadratic residue modulo $k$ and $0$ otherwise.
		So we are left with
		\begin{align*}
			d(s(n^2), q) &=\lim_{\lambda \to \infty} \sum_{m < k^{2 \lambda+1}} \ind_{[s(m) = q]} c(m; k^{2\lambda+1})\\
				 &= \lim_{\lambda \to \infty} \sum_{0\leq \mu \leq \lambda} \sum_{m'_0 = QR} \sum_{m' < k^{2\lambda-2\mu}} \ind_{[s(m' k^{2 \mu +1} + m'_0 k^{2\mu}) = q]} \frac{2}{k} \frac{1}{k^{\mu}} \frac{1}{k^{2\lambda-2\mu}}.
		\end{align*}
		Now we aim to split the contribution of $m'$ and $m'_0$.
		We let $\delta$ denote the transition function for $s$ and have $s(m' k^{2\mu +1} + m'_0 k^{2\mu}) = \delta(q_0, (m' k^{2 \mu +1} + m'_0 k^{2\mu})_k) = q$ if and only if there exist $q_1, q_2$ such that $\delta(q_0, (m')_k) = q_1, \delta(q_1, (m'_0)_k^{1}) = q_2$ and $\delta(q_2, (0)_k^{2 \mu}) = q$. Moreover, we see that
		\begin{align*}
			\frac{1}{k^{2\lambda - 2\mu}} \sum_{m' < k^{2\lambda-2\mu}} \ind{[\delta(q_0, (m')_k) = q]} = d(s(n), q) + o_{\lambda-\mu \to \infty}(1). 
		\end{align*}
		Thus we have in total
		\begin{align*}
			d(s&(n^2), q)\\
				 &= \lim_{\lambda \to \infty} \sum_{0\leq \mu \leq \lambda} \sum_{m'_0 = QR} \sum_{m' < k^{2\lambda-2\mu}} \ind_{[s(m' k^{2 \mu +1} + m'_0 k^{2\mu}) = q]} \frac{2}{k^{2\lambda-\mu +1}}\\
				&= \lim_{\lambda \to \infty} \sum_{q_1, q_2 \in Q} \sum_{0\leq \mu \leq \lambda} \sum_{m'_0 = QR} \sum_{m' < k^{2\lambda-2\mu}} \ind_{[\delta(q_0, (m')_k) = q_1]} \ind_{[\delta(q_1, (m'_0)_k^1) = q_2]} \ind_{[\delta(q_2, (0)_k^{2\mu}) = q]}\frac{2}{k^{2\lambda-\mu +1}}\\
				&=  \sum_{q_1, q_2 \in Q}\rb{\sum_{m'_0 = QR}  \ind_{[\delta(q_1, (m'_0)_k^1) = q_2]} \frac{2}{k}}\\
					&\qquad \cdot \lim_{\lambda \to \infty} \sum_{0\leq \mu \leq \lambda}  \frac{1}{k^{\mu}} \ind_{[\delta(q_2, (0)_k^{2\mu}) = q]} \cdot  \rb{ \frac{1}{k^{2\lambda-2\mu}}\sum_{m' < k^{2\lambda-2\mu}} \ind_{[\delta(q_0, (m')_k) = q_1]}}\\
				&=  \sum_{q_1, q_2 \in Q}\rb{\sum_{m'_0 = QR}  \ind_{[\delta(q_1, (m'_0)_k^1) = q_2]} \frac{2}{k}}\\
					&\qquad \cdot \lim_{\lambda \to \infty} \sum_{0\leq \mu \leq \lambda}  \frac{1}{k^{\mu}} \ind_{[\delta(q_2, (0)_k^{2\mu}) = q]} \cdot \rb{d(s(n), q_1) + o_{\lambda-\mu \to \infty}(1)}.
		\end{align*}
		We note that the sum of the $o(1)$ terms is negligible as $\sum_{\mu \geq 0} k^{-\mu}$ is absolutely convergent.
		Thus, we have
		\begin{align}\label{eq_density_synchronizing_squares}
			d(s(n^2), q) &=  \sum_{q_1, q_2 \in Q}d(s(n), q_1) \cdot  \sum_{m'_0 = QR}  \ind_{[\delta(q_1, (m'_0)_k^1) = q_2]} \frac{2}{k} \cdot \sum_{\mu \geq 0}  \frac{1}{k^{\mu}} \ind_{[\delta(q_2, (0)_k^{2\mu}) = q]}.
		\end{align}

		As $d(s, q_1) \in \Q$ it only remains to show that
		\begin{align}\label{eq_delta_02}
			\sum_{\mu \geq 0} \frac{1}{k^{\mu}} \ind_{[\delta(q_2, (0)_k^{2\mu}) = q]} \in \Q.
		\end{align}
		However, $\delta(q_2, (0)_k^{2\mu})$ is an eventually periodic sequence (as we iteratively apply $\delta(., (0)_k^2)$).
		Thus~\eqref{eq_delta_02} is a finite sum of geometric series with rational ratio and as such rational.

		The case $k = 2$ works mostly analogously. The only difference is that (due to Lemma~\ref{le_cor_hensel}) we have to write $m = m' k^{2\mu +3} + m'_0 k^{2\mu}$, where $m'_0 = 1$. 
		Hence, we have $c(m; 2^{2\lambda+1}) = c(2^{2\mu}; 2^{2\mu+3})/2^{2\lambda - 2\mu -2} = 1/2^{2 \lambda - \mu - 1}$. We find in total, analogously to~\eqref{eq_density_synchronizing_squares}
		\begin{align}\label{eq_density_squares_2}
			d(s(n^2), q) &=  \sum_{q_1, q_2 \in Q}d(s(n), q_1) \cdot \ind_{[\delta(q_1, (1)_2^3) = q_2]} \frac{1}{2} \cdot \sum_{\mu \geq 0}  \frac{1}{2^{\mu}} \ind_{[\delta(q_2, (0)_2^{2\mu}) = q]},
		\end{align}
		and the proof finishes using the same arguments.
	\end{proof}

\begin{example}
		We discuss again the paperfolding sequence with respect to Theorem~\ref{th_squares_primitive}.
		
		We first discuss the automatic sequence without the projection and call it $s'$\footnote{The output is then on the alphabet $a,b,c,d$.}.
		As the paperfolding sequence (and $s'$) is $2$-automatic, we have to apply~\eqref{eq_density_squares_2}. We also see that $\delta(q, 001) = b$ for all $q\in Q$. Thus, only $q_2 = b$ gives a positive contribution and clearly $\sum_{q_1\in Q} d(s'(n), q_1) = 1$. This gives
		\begin{align*}
			d(s'(n^2), q) = \frac{1}{2} \sum_{\mu \geq 0} \frac{1}{2^{\mu}} \ind{[\delta(b, (0)_2^{2\mu}) = q]}. 
		\end{align*}
		Moreover, we have $\delta(b,00) = a, \delta(a,00) = a$. So that only $a,b$ have a positive density along squares:
		\begin{align*}
			d(s'(n^2), a) &= \frac{1}{2} \sum_{\mu \geq 1} \frac{1}{2^{\mu}} = \frac{1}{2},\\
			d(s'(n^2), b) &= \frac{1}{2} \sum_{\mu = 0} \frac{1}{2^{\mu}} = \frac{1}{2}.
		\end{align*}
		As both $a$ and $b$ are projected to $1$, we find that the density of $1$ in the paperfolding sequence along squares is $1$.
	\end{example}

\subsection{Logarithmic densities of general automatic sequences}

We focus in this section on how to compute the logarithmic density of automatic sequences (in particular of the $M_i$ in Proposition~\ref{prop:decomposition}).
There is for example an explicit formula in~\cite{Allouche2003}, namely Theorem 8.4.8 (and Corollary 8.4.9). 
However, this one is rather hard to use in practical terms.
There is also a (slightly vague) description in a presentation by Bell~\cite{Bell_presentation}.
We can find a very similar (if not identical) description as in~\cite{Bell_presentation}:

We note that for each of the $M_i$ in Proposition~\ref{prop:decomposition} we have that if $m \in M_i$ then also $mk^{\lambda} + r \in M_i$ for all $\lambda \geq 0, 0\leq r < k^{\lambda}$.
Thus, we let $S_i$ denote the set of integers that "generate" $M_i$, i.e.
\begin{align*}
	S_i := \{m \in M_i: \not \exists m_0 \in M_i, \lambda \geq 0, 0\leq r < k^{\lambda} \mbox{ with } m = m_0k^{\lambda} + r\}.
\end{align*}
This allows us to decompose $M_i$ into a disjoint union,
\begin{align*}
	M_i = \bigcup_{m \in S_i} \{mk^{\lambda} + r: \lambda \geq 0, 0\leq r < k^{\lambda}\}.
\end{align*}

A simple computation shows
\begin{align*}
	d_{\log}(\{mk^{\lambda} + r: \lambda \geq 0, 0\leq r < k^{\lambda}\}) &= \lim_{L \to \infty} \frac{1}{\log(m k^{L})} \sum_{0\leq \lambda < L} \sum_{0\leq r < k^{\lambda}} \frac{1}{mk^{\lambda} + r}\\
		&= \lim_{L \to \infty} \frac{1}{\log(m k^{L})} \sum_{0\leq \lambda < L} \log\rb{\frac{(m+1)k^{\lambda}}{mk^{\lambda}}} + O\rb{\frac{1}{mk^{\lambda}}}\\
		&= \lim_{L \to \infty} \frac{1}{L \log(k) + \log(m)} \rb{L \log(1 + 1/m) + O(1)}\\
		&= \frac{\log(1+1/m)}{\log(k)}.
\end{align*}
This gives in total
\begin{align}\label{eq_d_log_M}
	d_{\log}(M_i) = \frac{1}{\log(k)}\sum_{m \in S_i} \log(1+1/m).
\end{align}

There is a conjecture in~\cite{Allouche2003} that says that the logarithmic density is always the fraction of logarithms of rational numbers.
This is trivially true whenever $S_i$ is finite. However, there are also examples where this is not obvious at all (the following example also appeared in~\cite{Bell_presentation}):
\begin{example}
	We consider the following $3$-automatic sequence that is $1$ if the base $3$ expansion starts with $100\ldots001$ and $0$ otherwise. The corresponding automaton is given below.
	
	
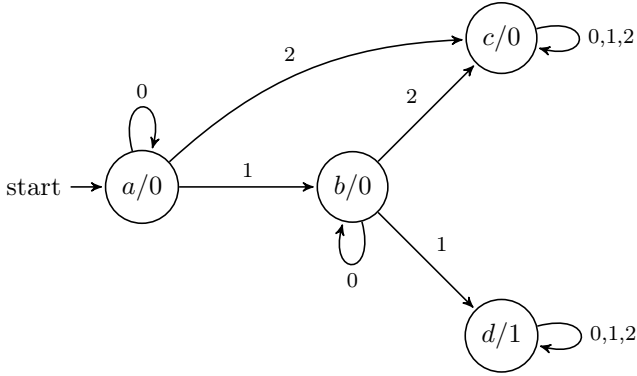
\begin{tikzpicture}[->,>=stealth',shorten >=1pt,auto,node distance=2.8cm, semithick, bend angle = 20, cross line/.style={preaction={draw=white,-,line width=4pt}}]
    
    \node[state, initial](A)                    {$a/0$};
    \node[state]         (B) [right of=A] 			{$b/0$};
    \node[state]	(C) [above right of =B]		{$c/0$};
    \node[state]	(D) [below right of =B]		{$d/1$};
    
    \path [every node/.style={font=\footnotesize}, pos = 0.5]
    (A) edge 		  node 		 {1}      (B)
	edge [loop above] node		{0} 	(A)
	edge [bend left]	node		{2}	(C)
    (B) edge 			node 		{1} 	(D)
    	edge			node		{2}	(C)
	edge [loop below] node 		{0}	(B)
    (C) edge [loop right] 	node		{0,1,2} (C)
    (D) edge [loop right] 	node		{0,1,2} (D);
	\end{tikzpicture}

We have that $b_1(n) = 0$ and $b_2(n) = 1$ for all $n\geq 0$. Moreover one finds that $S_2 = \{3^{\lambda} +1: \lambda \geq 1\}$.

Thus one has
\begin{align*}
	d_{\log}(M_2) &= \frac{1}{\log(3)} \sum_{\lambda \geq 1} \log\rb{1 + \frac{1}{3^{\lambda} + 1}}\\
		&= \frac{1}{\log(3)} \log\rb{\prod_{\lambda \geq 1} \rb{1 + \frac{1}{3^{\lambda} + 1}}}.
\end{align*}
\end{example}

We end this section with another example that was already discussed in~\cite{Muellner2017} for which the density along primes does not exist.
\begin{example}
We consider the following automaton and the corresponding automatic sequence $(a(n))_{n\in\N}$.

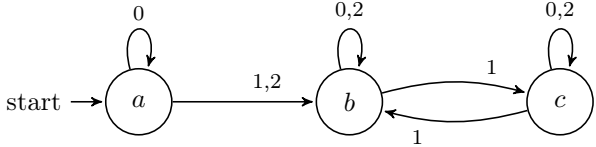
  \begin{tikzpicture}[->,>=stealth',shorten >=1pt,auto,node distance=2.8cm, semithick, bend angle = 15, cross line/.style={preaction={draw=white,-,line width=4pt}}]
    
    \node[state, initial](A)                    {$a$};
    \node[state]         (B) [right of=A] 			{$b$};
    \node[state]         (C) [right of=B] 			{$c$};
    
    \path [every node/.style={font=\footnotesize}, pos = 0.66]
    (A) edge 		  node 		 {1,2}      (B)
	edge [loop above] node [pos=0.5] {0} 	(A)
    (B) edge [bend left]  node 		 {1}	(C)
	edge [loop above]  node 	[pos=0.5]	 {0,2}	(B)
    (C) edge [bend left] node 		 {1} 	(B)
	edge [loop above] node 		[pos=0.5] 	{0,2}	(C);
	\end{tikzpicture}
	
  It follows by the discussion in~\cite{Muellner2017} that $a(n) = b$ holds in exactly two cases:
  \begin{itemize}
   \item $n$ is even and the first digit of $n$ in base $3$ is $2$,
   \item $n$ is odd and the first digit of $n$ in base $3$ is $1$.
  \end{itemize}
  
  One finds easily that the $a(n)$ is equally distributed on $\{b,c\}$, i.e. $d(a(n), b) = d(a(n), c) = 1/2$.
 But as discussed in~\cite{Muellner2017} the density of $b$ and $c$ do not exist along primes.

Now how does this example work in light of Theorem~\ref{Th2} (and Proposition~\ref{prop:decomposition})?
We first find a decomposition as in Proposition~\ref{prop:decomposition}. Therefore, let
\begin{align*}
\begin{split}
	b_1(n) = 
\left\{\begin{array}{cl} b, & \mbox{if } n \text{ is odd},\\c & \mbox{otherwise}, \end{array}\right.
\end{split}
\begin{split}
	b_2(n) = 
\left\{\begin{array}{cl} c, & \mbox{if } n \text{ is odd},\\b & \mbox{otherwise}, \end{array}\right. 
\end{split}
\end{align*}
and $M_i$ ($i = 1,2$) denotes the set of integers for which the first digit in base $3$ is $i$. One finds directly by the discussion above that this choice satisfies Proposition~\ref{prop:decomposition}. As all prime numbers (except $2$) are odd we have directly $d(b_1(p_n), b) = d(b_2(p_n), c) = 1$ and $d(b_1(p_n), c) = d(b_2(p_n), b) = 0$. 
Moreover, we see that $S_1 = \{1\}$ and $S_2 = \{2\}$. Thus we see by~\eqref{eq_d_log_M} that $d_{\log}(M_1) = \log(2)/\log(3)$ and $d_{\log}(M_2) = (\log(3)-\log(2))/\log(3)$.
	This shows with~\eqref{eq_d_log} that
	\begin{align*}
		d_{\log}(a(p_n), b) = \frac{\log(2)}{\log(3)} \cdot 1 + \frac{\log(3)-\log(2)}{\log(3)} \cdot 0 = \frac{\log(2)}{\log(3)}.
	\end{align*}
	\end{example}

\appendix
\section{Implications for dynamical systems}\label{sec:dynamical}

The decomposition of an automatic sequence in primitive and prolongable automatic sequences in Proposition~\ref{prop:decomposition} has an interesting counterpart in the world of dynamics.
We start off with a short introduction to dynamical systems associated with sequences.

There is a long history for considering dynamical systems associated with sequences (see for example~\cite{Queffelec2010}, which is especially concerned with automatic sequences).\footnote{In this context one works with substitutions of constant length instead of automata. However, we will try to avoid introducing different concepts if not strictly necessary.}

We first define the \emph{language} of a sequence $\mathbf{u} = (u(n))_{n\in \N}$ (or $\Z$ instead of $\N$) taking values in a finite alphabet $\mathbb{A}$ as
\begin{align*}
	\mathcal{L}(\mathbf{u}) := \{u(m)\cdots u(n): m\leq n\},
\end{align*}
i.e. the language is the set of all non-empty factors of $\mathbf{u}$.
Then we can associate a compact set with this sequence,
\begin{align*}
	X_{\mathbf{u}} := \{x \in \mathbb{A}^{\Z}: \mathcal{L}(x) \subseteq \mathcal{L}(\mathbf u)\} = \overline{\{(u(n+\ell))_{n\in \N}: \ell \in \N\}}.
\end{align*}
That is the minimal compact set containing $\mathbf{u}$, that is closed under the shift $T$, where
\begin{align*}
	T((x(n))_{n\in \N} := (x(n+1))_{n\in \N}.
\end{align*}
Therefore, $(X_{\mathbf{u}}, T)$ it is a canonical candidate to consider, when one wants to use methods or ideas coming from dynamical systems.

It proved to be useful to consider two-sided sequences ($\Z$) instead of one-sided sequences ($\N$) for the case when $\mathbf{u}$ is an automatic sequence.

We can make $(X_{\mathbf{u}}, T)$ a topological dynamical system by using the metric
\begin{align*}
	d(x, y) = \sum_{n\geq 0} \frac{1}{2^{n+2}} (d_n(x_n, y_n) + d_{-n}(x_{-n}, y_{-n})),
\end{align*}
where $d_n$ denotes the discrete metric on $\mathbb{A}$.

We can also consider a \emph{measure-theoretic dynamical system}, i.e. $(X_{\mathbf{u}}, \mathcal{B}, \mu, T)$ where $(X, \mathcal{B}, \mu)$ is a standard Borel probability space and $T: X_{\mathbf{u}} \to X_{\mathbf{u}}$ is an a.e. bijection which is bimeasurable and measure-preserving.
We call $(X_{\mathbf{u}}, \mathcal{B}, \mu, T)$ \emph{ergodic} if for every $E\in \mathcal{B}$ with $T^{-1}(E) = E$ follows either $\mu(E) = 0$ or $\mu(E) = 1$.

Each homeomorphism $T$ of a compact metric space $X$ determines many (measure-theoretic) dynamical systems
$(X,\mathcal{B}(X),\mu,S)$ with $\mu\in M(X,S)$, where $M(X,T)$ stands
for the set of Borel probability measures on $X$ ($\mathcal{B}(X)$
stands for the $\sigma$-algebra of Borel sets of $X$).
Recall that by the Krylov-Bogolyubov theorem, $M(X,T)\neq\emptyset$,
and moreover, $M(X,T)$ endowed with the weak-$\ast$ topology becomes a
compact metrizable space. The set $M(X,T)$ has a natural structure of
a convex set (in fact, it is a Choquet simplex) and its extremal
points are precisely the ergodic measures.
We say that the topological system $(X,T)$ is {\em uniquely ergodic}
if  it has only one invariant measure (which must be ergodic).
The system $(X,T)$ is called {\em minimal} if it does not contain a
proper subsystem (equivalently, the orbit of each point is dense).
Furthermore, a point $x \in X$ is called an \emph{almost periodic} point if for any neighborhood $U$ of $x$ there exists $N \in \N$ such that 
\begin{align*}
	\{T^{n+i}(x): i = 0, \ldots, N\} \cap U \neq \emptyset,
\end{align*}
for all $n \in \N$.

It is a classical result that if $\mathbf{u}$ is a primitive and prolongable automatic sequence, then $(X_{\mathbf{u}}, T)$ is \emph{strictly ergodic}, that is minimal and uniquely ergodic. Moreover, every point $x \in X_{\mathbf{u}}$ is almost periodic.

\begin{lemma}\label{le_dynamics}
	Assume that $(X,T)$ is a topological dynamical system and let $x$ be an almost periodic point and $y \in X$ for which $d(T^{j_n}x,T^{i_n}y)\to0$ when $n\to\infty$. Then $\overline{\{T^kx:\:k\in\Z\}}\subset \overline{\{T^ky:\:k\in\Z\}}$.
\end{lemma}
  \begin{proof}
	By passing to a subsequence $T^{j_{n_s}}x\to x'$ and $T^{i_{n_s}}y\to y'$, where necessarily $x'=y'$. This shows that the intersection of the closures of the two orbits is non-empty, so the claim follows from minimality of the orbit closure of $x$.
\end{proof}
\begin{remark}
	The condition $d(T^{j_n}x,T^{i_n}y)\to0$ when $n\to\infty$ is equivalent to the fact that $x$ and $y$ have arbitrarily long common subwords.
\end{remark}

\begin{corollary}
	With the notation of Proposition~\ref{prop:decomposition}, we have that $X_{b_i} \subset X_a$ for all $i$. Furthermore, we have for all $i,j$, either $X_{b_i} = X_{b_j}$ or $X_{b_i} \cap X_{b_j} = \emptyset$.
\end{corollary}
\begin{proof}
	It follows from Proposition~\ref{prop:decomposition} that $a$ and $b_i$ coincide on arbitrarily long intervals. Thus the condition of Lemma~\ref{le_dynamics} is fulfilled and the desired result follows. The second result follows easily as both $X_{b_i}, X_{b_j}$ are minimal.
\end{proof}

\begin{proposition}\label{pr_minimal_dynamics}
	Each automatic sequence $a$ yields a subshift $X_a$ which has only finitely many minimal components. They are given by the $X_{b_i}$. 
\end{proposition}
\begin{proof}
	This is in its essence only a reformulation of Proposition 2.2 in~\cite{Byszewski2019}. However, we provide nevertheless a proof as it highlights important ideas for the proof of Proposition~\ref{pr_measure_dynamics}.
	
	First we note that there exists some $\ell \in \N$ such that every consecutive $\ell$ integers contain an integer $n_0 \notin M_0$, as otherwise the upper Banach density of $M_0$ would be $1$.
	Let us now assume that $z$ is an almost periodic point in $X_a$. Fix $K = k^{\lambda} (\ell +2)\geq1$ for some $\lambda$ (we will later let $\lambda \to \infty$) and we find
\begin{align*}
	z(0)z(1)\cdots z(K) = a(L)a(L+1)\cdots a(L+K),
\end{align*}
for some $L \in \N$ as $\mathcal{L}(z) \subset \mathcal{L}(a)$.

We find by our definition of $K$ that $I = [L/k^{\lambda}, (L+K)/k^{\lambda}-1]$ contains at least $\ell$ consecutive integers, so that there exists $n_0 \in I$ with $n_0 \notin M_0$.
Thus, we have $n_0 \in M_i$ for some $i \geq 1$ and by the properties of $M_i$ also that $n_0 k^{\lambda} + r \in M_i$ for all $0\leq r < k^{\lambda}$.

Thus we find that for every $\lambda \in \N$ there exists some $i\geq 1$ such that $z$ and $b_i$ have a common subword of length $k^{\lambda}$. As there are only finitely many $b_i's$ there has to exist some $i_0 \geq 1$ such that $z$ and $b_{i_0}$ have arbitrarily long common subwords and we can apply Lemma~\ref{le_dynamics}.
This shows that $\overline{\{T^k(z):\:k\in\Z\}} = \overline{\{T^k(b_{i_0}):\:k\in\Z\}}$ as both $z$ and $b_{i_0}$ are almost periodic.
\end{proof}

\begin{proposition}\label{pr_measure_dynamics}
	The only ergodic measures in $X_a$ are given by the unique measures determined by $X_{b_i}$ ($i\geq1$). (In other words the ergodic decomposition is in a sense a decomposition into minimal components.) 
\end{proposition}
\begin{proof}
	Indeed, if $z$ is a generic point for an ergodic measure $\nu$\footnote{A point $x$ is called generic for a measure $\nu$ if $\lim \limits_{n\to \infty} \frac{1}{n} \sum_{i = 0}^{n-1} f(T^i(x)) = \int_X f d\nu$ holds for all $f \in C(X)$, whose existence is guaranteed by the ergodic Theorem.} then similarly to the proof of Proposition~\ref{pr_minimal_dynamics} we let $K = (\ell +2) \cdot k^{\lambda}$, where we let this time both $\ell \to \infty$ and $\lambda \to \infty$.
We find by the same reasoning as before that
\begin{align*}
	z(0)z(1)\ldots z(K) = a(L)a(L+1)\ldots a(L+K),
\end{align*}
for some $L \in \N$ and find that $I = [L/k^{\lambda}, (L+K)/k^{\lambda}-1]$ contains at least $\ell$ consecutive integers. As the upper Banach density of $M_0$ is zero, we know that the proportion of integers in $I$ that do belong to $M_0$ tend to $0$ as $\ell \to \infty$. Thus, we can cover $[L, K]$ by blocks of the $b_i$ of length $k^{\lambda}$ (up to a small proportion). It follows that $\nu$ is supported by the union of supports of the unique measures given by the $b_i$'s. Since $\nu$ has to be positive on some $X_{b_i}$. As there is only one ergodic measure on $X_{b_i}$ it follows that $\nu$ has to coincide with it.
\end{proof}

\begin{remark}
	Proposition~\ref{pr_minimal_dynamics} and Proposition~\ref{pr_measure_dynamics} show that the ergodic decomposition of invariant measures of  dynamical systems associate with automatic sequences actually corresponds to the decomposition of the topological dynamical system into minimal components.
\end{remark}

Lastly, we give a short application of this decomposition.
\begin{corollary}\label{cor:ortho}The subshift $(X_a,S)$ generated by any automatic sequence $a$  is orthogonal to any bounded multiplicative aperiodic function.
\end{corollary}
\begin{proof}
We take any point $y\in X_a$ and suppose that it is quasi-generic for a measure $\nu$. Its ergodic decomposition consists of finitely many measures, each of which yields a system $(X_{b_i}, T)$ which satisfies the strong MOMO property by~\cite[Lemma 8.1]{Lemanczyk2018}. So the result follows from a general theory (see for example~\cite[Theorem 4.1]{Kanigowski2019}).
\end{proof}

\section{Upper density}\label{sec_upper}

This section is devoted to sketching a proof for the following theorem which is a generalization of the corresponding result in~\cite{Bell2020} and follows in large parts the same ideas.

\begin{theorem}\label{th_upper_density}
	Under the same conditions as in Theorem~\ref{Th1} together with the additional assumption $\beta = 1$, the upper and lower densities,
	\begin{equation*}
		\overline{d}(a(n_{\ell}), \alpha) = \limsup_{x\to \infty} \frac{1}{x} \sum_{\ell \leq x} \inda 
		\;\;\;\mbox{ and } \;\;\;
		\underline{d}(a(n_{\ell}), \alpha) = \liminf_{x\to \infty} \frac{1}{x} \sum_{\ell \leq x} \inda \,,
	\end{equation*}
	can be explicitly computed.
	Moreover, if the densities for primitive and prolongable automatic sequences $\tilde{a}(n)$ along the subsequence $(n_{\ell})$ are rational, then so are the upper and lower densities of $a(n)$ along the subsequence $(n_{\ell})$.
\end{theorem}

In particular, we can apply this theorem to the subsequence along primes.
\begin{corollary}\label{co_upper_primes}
	Let $(a(n))_{n\in \N}$ be an automatic sequence.
	Then, the upper and lower densities along the primes are rational and can be explicitly computed.
\end{corollary}

We will need the following preliminary results.
\begin{lemma}[Lemma 3.1 in~\cite{Bell2020}]\label{le_upper_density}
	Let $k \geq 2$ be a natural number, let $\gamma$ be a positive real number, let $s_n, s_n'$ be sequences of non-negative numbers, let $u', v', u$ be non-negative real numbers, $v$ a positive real number, and let $b$ and $c$ be positive integers. If 
	\begin{align*}
		\frac{(v'k^{b+c} + u' k^c + s_n')}{(v k^{b+c} + uk^{c} + s_n)} \to \gamma
	\end{align*}
	as $n \to \infty$ and
	\begin{align*}
		\limsup_{n\to\infty} \frac{(v'k^{c} + s_n')}{(v k^{c} + s_n)} \leq \gamma,
	\end{align*}
	then
	\begin{align*}
		\limsup_{n\to\infty} \frac{(v'k^{2b+c} + u' (k^{b+c}+k^c) + s_n')}{(v k^{2b+c} + u(k^{b+c} +k^{c}) + s_n)} \geq \gamma.
	\end{align*}
\end{lemma}
We note that Lemma~\ref{le_upper_density} was stated in~\cite{Bell2020} for positive sequences $s_n, s_n'$ and positive numbers $u', v', u$, but the proof only requires them to be non-negative.

From now on, we let $(a(n))_{n\in\mathbb N}$ denote a $k$-automatic sequence and we let $M_i$, $i\leq 1\leq s$, be the sets defined as in 
Proposition \ref{prop:decomposition}. We recall that in general the density of the automatic sets $M_i$ do not exist. 
However, the density exists if we only consider certain intervals.

\begin{lemma}\label{le_density_M_i}
	There exists an automaton $\mathcal{A} = (Q, \Sigma_k, \delta, q_0)$ such that for $1 \leq i \leq s$ and $m \in \N$,
	\begin{align}\label{eq_limit_M_i}
		\lim_{\nu \to \infty} \frac{\abs{M_i \cap [mk^{\nu}, (m+1)k^{\nu}-1]}}{k^{\nu}} = d_{i,\delta(q_0, (m)_k)}
	\end{align}
	exists and is rational.
\end{lemma}
\begin{proof}
	We recall that if $n \in M_i$ then also $n\cdot k + j \in M_i$, for any $j \in \{0,\ldots, k-1\}$.
	This shows that the sequence $ \frac{\abs{M_i \cap [mk^{\nu}, (m+1)k^{\nu}]}}{k^{\nu}}$ is monotonously increasing and, obviously, bounded by $1$.
	Thus, the limit exists.
	
	Moreover, we recall that $M_i$ is a $k$-automatic set and, thus, there exists an automaton 
	$$\mathcal{A}^{(i)} = (Q^{(i)}, \Sigma_k, \delta^{(i)}, q_0^{(i)}, \{0,1\}, \tau^{(i)})$$ such that $\ind_{[n \in M_i]} = \tau^{(i)}(\delta^{(i)}(q_0^{(i)}, (n)_k))$.
	We define $\mathcal{A}$ as the product of these $s$ automata, i.e. $$\mathcal{A} = (Q, \Sigma_k, \delta, q_0) = (Q^{(1)} \times \cdots \times Q^{(s)}, \Sigma_k, \delta^{(1)} \times \cdots \times \delta^{(s)}, (q_0^{(1)}, \ldots, q_0^{(s)}))\, .$$
	Let $\tau_i := \tau^{(i)} \circ \pi_i$, where $\pi_i$ is the projection to the $j$-th coordinate. 
	It follows directly that $\tau^{(i)} (\delta^{(i)}(q_0^{(i)}, (n)_k)) = \tau_i(\delta(q_0, (n)_k))$.
	
	We assume without loss of generality that $Q = \{q_0, q_1, \ldots, q_d\}$ and use the incidence matrix $A$\footnote{See also Section~\ref{sec:computability} for the definition of the incidence matrix. We let $A$ (instead of $M$) denote here the incidence matrix in order to avoid any confusion with the sets $M_i$.} of $\mathcal{A}$ to compute the limit in Equation~\eqref{eq_limit_M_i}.
	We see directly, that $mk^{\nu} + r \in M_i$ if and only if $\tau_i(\delta(\delta(q_0, (m)_k), (r)_k^{\nu})) = 1$.
	When we consider the sum over $r<k^{\nu}$, we can use the incidence matrix as it encodes the sum of all possible transitions by words of length $1$.
	Indeed, we find
	\begin{align}\label{eq_something}
		\frac{\abs{M_i \cap [mk^{\nu}, (m+1)k^{\nu}-1]}}{k^{\nu}} = v_{\tau_i} \cdot (A/k)^{\nu} \cdot e_{\delta(q_0, (m)_k)},
	\end{align}
	where $v_{\tau_i} = (\tau_i(q_0), \tau_i(q_1), \ldots, \tau_i(q_d))^T$ and $e_{q_j}$ denotes the $j$-th unit vector (here we use the convention that the indices start with $0$).
	
	It is easy to see that $1$ is the largest eigenvalue of $A/k$, but it can happen that it is not the unique eigenvalue of $A/k$ with absolute value $1$.
	The limit on the right-hand side of Equation~\eqref{eq_something} can be dealt with as in the proof of~\cite[Proposition 2.1]{Bell2020} which shows that \eqref{eq_something} converges to a rational number along a subsequence, since the output of $\mathcal{A}$ is always rational. This finishes the proof as we know that the limit exists.
\end{proof}

\begin{corollary}
	We have for any $q \in Q$
	\begin{align*}
		d_{1,q} + \ldots + d_{s,q} = 1.
	\end{align*}
\end{corollary}
\begin{proof}
	This follows directly from Lemma~\ref{le_density_M_i} and the fact that the density of $M_0$ exists and equals $0$.
\end{proof}

Lemma~\ref{le_density_M_i} allows us to compute the density of $\inda$ in the same intervals, as long as $n_{\ell}$ is slowly varying.

\begin{proposition}\label{pr_evaluate_m}
	Let $m \in \N$ and assume that $\beta = 1$. Then
	\begin{align*}
		\lim_{\nu \to \infty}\frac{1}{g(k^{\nu})} \sum_{g(mk^{\nu})\leq \ell < g((m+1)k^{\nu})} \inda = \sum_{1\leq i \leq s} d_{i,\delta(q_0, (m)_k)} \cdot d(b_i(n_{\ell}), \alpha).
	\end{align*}
\end{proposition}
\begin{proof}

	Let $\varepsilon >0$ and choose $\lambda \in \N$ such that 
	\begin{align*}
		\abs{\frac{\abs{M_i \cap [mk^{\lambda}, (m+1)k^{\lambda}-1]}}{k^{\lambda}} - d_{i,\delta(q_0, (m)_k)}} \leq \varepsilon,
	\end{align*}
	for $i = 1, \ldots, s$ according to Lemma~\ref{le_density_M_i}.
	This allows us to write
	\begin{align*}
		\lim_{\nu \to \infty}&\frac{1}{g(k^{\nu})} \sum_{g(mk^{\nu})\leq \ell < g((m+1)k^{\nu})} \inda\\
		 &= \lim_{\nu \to \infty}\frac{1}{g(k^{\nu+\lambda})} \sum_{r < k^{\lambda}} \sum_{g(mk^{\nu + \lambda} + r k^{\nu})\leq \ell < g(mk^{\nu+\lambda} + (r+1) k^{\nu} )} \inda\\
		 &= \lim_{\nu \to \infty} \frac{g(k^{\nu})}{g(k^{\nu + \lambda})} \sum_{0\leq i \leq s} \sum_{\substack{r < k^{\lambda}\\ mk^{\lambda} + r \in M_i}} \frac{1}{g(k^{\nu})} \sum_{g(mk^{\nu + \lambda} + r k^{\nu})\leq \ell < g(mk^{\nu+\lambda} + (r+1) k^{\nu} )} \inda.
	\end{align*}
	Obviously, $g(k^{\nu})/g(k^{\nu+\lambda})$ converges to $k^{-\lambda}$, as $g$ is regularly varying. In the last sum, we can replace for $i \neq 0$, $a$ by $b_i$. Moreover, theses sums (together with the factor $1/g(k^{\nu})$) then converge to
	\begin{align*}
		\rb{(m k^{\lambda} + r+1)^{\beta} - (mk^{\lambda} + r)^{\beta}} d(b_i(n_{\ell}), \alpha),
	\end{align*}
	by~\eqref{eq_density_interval}.
	We recall that that the number of $r \in M_i \cap [mk^{\lambda}, (m+1) k^{\lambda}-1]$ is approximately $k^{\lambda} \cdot d_{i,\delta(q_0,(m)_k)}$.

	This gives in total,
	\begin{align*}
		\lim_{\nu \to \infty}&\frac{1}{g(k^{\nu})} \sum_{g(mk^{\nu})\leq \ell < g((m+1)k^{\nu})} \inda\\
		 &= \sum_{1\leq i \leq s} \sum_{\substack{r < k^{\lambda}\\ mk^{\lambda} + r \in M_i}} d(b_i(n_{\ell}), \alpha) + O(\abs{M_0 \cap [mk^{\lambda}, (m+1) k^{\lambda}-1]})\\
		&= \sum_{1\leq i \leq s} d_{i,\delta(q_0, (m)_k)} \cdot d(b_i(n_{\ell}), \alpha) + O(\varepsilon),
	\end{align*}
	which finishes the proof as $\varepsilon$ was arbitrary.
\end{proof}
\begin{remark}
	Most of the proof of Proposition~\ref{pr_evaluate_m} also works for $\beta >1$. The main difficulty seems to be the evaluation of
	\begin{align*}
		\frac{1}{k^{\lambda \beta}} \sum_{\substack{r <k^{\lambda}\\ mk^{\lambda} + r \in M_i}}(m k^{\lambda} + r+1)^{\beta} - (mk^{\lambda} + r)^{\beta}.
	\end{align*}
	It is clear that the limit for $\lambda \to \infty$ exists, as it is monotonously increasing and bounded, but showing the rationality of the limit seems to be much harder than for $\beta = 1$.
\end{remark}

Now we are ready to tackle the proof of Theorem~\ref{th_upper_density}, very similarly to the corresponding proof in~\cite{Bell2020}.

\begin{proof}[Proof of Theorem~\ref{th_upper_density}]
We note that the Theorem actually holds for any output function of the original automatic sequence, as long as the values are rational numbers. For simplicity of notation, we will only consider the case of $\inda$.

We put for $m \in \N$,
\begin{align*}
	\gamma(m) := \limsup_{\nu \in \N, r < k^{\nu}} \frac{1}{g(mk^{\nu} + r)} \sum_{\substack{\ell \in \N\\ n_{\ell} \leq mk^{\nu} + r}} \inda.
\end{align*}
In particular $\gamma(0) = \bar{d}(a(n_{\ell}), \alpha)$.
We let $\mathcal{S}$ denote the set of integers $m$ such that $\gamma(m) =  \bar{d}(a(n_{\ell}), \alpha)$.
It is easy to show that $\mathcal{S}$ is infinite, as for $m \in \mathcal{S}$ at least one of $k m + r$ for $r = 0, \ldots, k-1$ belongs to $\mathcal{S}$ as well.

Let $\mathcal{A}$ be as in the conclusion of Lemma~\ref{le_density_M_i}.
It follows from the pigeonhole principle that, if $m'\in \mathcal{S}$ is large enough, we can decompose it as $m' = m_1 k^{\lambda_2 + \lambda_3} + m_2 k^{\lambda_3} + m_3$, where $\lambda_i \in \N_{>0}, 0 \leq m_i < k^{\lambda_i}$ for $i = 1,2,3$, $m_1 > 0$ and
\begin{align}\label{eq_q_1}
	\delta(q_0,(m_1)_k) = \delta(q_0, (m_1 k^{\lambda_2} +m_2)_k) =:q_1.
\end{align}
It follows from the definition of $\mathcal{S}$ and $\gamma$ that also $m := m_1k^{\lambda_2} + m_2 \in \mathcal{S}$.
Our goal is to show that $m^{(h)} \in \mathcal{S}$ for any $h \in \N$, where
\begin{align}
	m^{(h)} := m_1 k^{h \lambda_2} + m_2 \frac{k^{h \lambda_2}-1}{k^{\lambda_2}-1}\in \mathcal{S}.
\end{align}
That is, we want to show that we can repeat the loop corresponding to $m_2$ (c.f. \eqref{eq_q_1}) while remaining in the set $\mathcal{S}$.
This makes intuitively sense, as $m_2$ needs to ``maximize its contribution'' corresponding to $\gamma(m_1k^{\lambda_2} + m_2) = \bar{d}(a(n_{\ell}), \alpha)$.
The technical problem (and reason for Lemma~\ref{le_upper_density}) is the dependency of the normalizing weight on $m_2$.

To show that $m^{(h)} \in \mathcal{S}$ for any $h \geq 1$, we will first use Lemma~\ref{le_upper_density} to show that $m^{(2)} \in \mathcal{S}$ and then use an inductive argument. From this we will finish the proof easily.

We can write
\begin{align*}
	\frac{1}{g(k^{\nu})}\sum_{\substack{\ell \in \N\\ n_{\ell} \leq mk^{\nu} + r}}  \inda &= \frac{1}{g(k^{\nu})}\sum_{\ell < g(m_1 k^{\nu + \lambda_2})} \inda + \frac{1}{g(k^{\nu})}\sum_{\ell = g(m_1 k^{\nu + \lambda_2})}^{g(m  k^{\nu})-1} \inda \\
		&\qquad + \frac{1}{g(k^{\nu})}\sum_{\ell = g(m k^{\nu})}^{g(m k^{\nu} + r)} \inda.
\end{align*}
Using Proposition~\ref{pr_evaluate_m} to rewrite the first sum for $\nu \to \infty$ gives
\begin{align*}
	\frac{1}{g(k^{\nu})}\sum_{\ell < g(m_1 k^{\nu + \lambda_2})} \inda &= \sum_{m_1' < m_1} \frac{1}{g(k^{\nu})}\sum_{\ell = g(m_1' k^{\nu+\lambda_2})}^{g((m_1'+1)k^{\nu+\lambda_2})-1} \inda\\
		&\to_{\nu \to \infty} k^{\lambda_2} \underbrace{\sum_{1\leq i \leq s} d(b_i(n_{\ell}), \alpha) \sum_{m_1' < m_1} d_{i, \delta(q_0, (m_1')_k)}}_{:= \kappa_1} + o(1).
\end{align*}
For the second sum we, furthermore, use Equation~\eqref{eq_q_1} to find
\begin{align*}
	\frac{1}{g(k^{\nu})}\sum_{\ell = g(m_1 k^{\nu + \lambda_2})}^{g(m k^{\nu})-1} \inda &\to_{\nu \to \infty} \sum_{1\leq i \leq s} d(b_i(n_{\ell}), \alpha) \sum_{m_2' <m_2} d_{i, \delta(q_0, (m_1 k^{\lambda_2} + m_2')_k)} + o(1)\\
		&= \underbrace{\sum_{1\leq i \leq s} d(b_i(n_{\ell}), \alpha) \sum_{m_2' <m_2} d_{i, \delta(q_1, (m_2')_k^{\lambda_2})}}_{:= \kappa_2} + o(1).
\end{align*}

This, finally, allows us to write
\begin{align*}
	\bar{d}(a(n_{\ell}), \alpha) &= \limsup_{\nu \in \N, r< k^{\nu}} \frac{\sum_{\ell \leq g(mk^{\nu} + r)} \inda}{g(mk^{\nu} + r)}\\
		&= \limsup_{\nu \in \N, r< k^{\nu}} \frac{\frac{1}{g(k^{\nu})} \sum_{\ell \leq g(mk^{\nu} + r)} \inda}{g(mk^{\nu} + r)/g(k^{\nu})}\\
		&= \limsup_{\nu \in \N, r< k^{\nu}} \frac{k^{\lambda_2}\kappa_1 + \kappa_2 + \frac{1}{g(k^{\nu})} \sum_{\ell = g(m k^{\nu})}^{g(mk^{\nu} + r)} \inda}{m_1k^{\lambda_2} + m_2 + \frac{r}{k^{\nu}}}.
\end{align*}
Here we used the fact that the $\lim_{N \to \infty} g(\delta N)/g(N)$ converges uniformly to $\delta^{\beta}$ in any compact interval.

Thus, we can choose $\nu_j \in \N, r_j < k^{\nu_j}$ for $j \in \N$, such that $\nu_j \to \infty$ for $j\to \infty$ and
\begin{align*}
	\bar{d}(a(n_{\ell}), \alpha)= \lim_{j \to \infty} \frac{k^{\lambda_2}\kappa_1 + \kappa_2 + \frac{1}{g(k^{\nu_j})} \sum_{\ell = g(m k^{\nu_j})}^{g(mk^{\nu_j} + r_j)} \inda}{m_1k^{\lambda_2} + m_2+ \frac{r_j}{k^{\nu_j}}}.
\end{align*}
  	
Next, we aim to use Lemma~\ref{le_upper_density}. Therefore, we consider
\begin{align}\label{eq_shortened}
	\limsup_{j \to \infty} \frac{1}{g(m_1 k^{\nu_j} +  r_j)} \sum_{\ell \leq g(m_1 k^{\nu_j} + r_j)} \inda
\end{align}
which can be computed analogously to the computation above and equals
\begin{align*}
	 \limsup_{j \to \infty} \frac{\kappa_1 + \frac{1}{g(k^{\nu_j})} \sum_{\ell = g(m_1 k^{\nu_j})}^{g(m_1 k^{\nu_j} + r_j)} \inda}{m_1 + \frac{r_j}{k^{\nu_j}}}.
\end{align*}
We are already almost in a position to use Lemma~\ref{le_upper_density}. The only problem is the difference between
\begin{align}\label{eq_sn_1}
	\limsup_{j\to\infty} \frac{1}{g(k^{\nu_j})} \sum_{\ell = g(m k^{\nu_j})}^{g(mk^{\nu_j} + r_j)} \inda
\end{align}
and
\begin{align}\label{eq_sn_2}
	\limsup_{j\to \infty} \frac{1}{g(k^{\nu_j})} \sum_{\ell = g(m_1 k^{\nu_j})}^{g(m_1 k^{\nu_j} + r_j)} \inda.
\end{align}

However, both of these sums can be approximated as follows. Fix any $\lambda \in \N$ and let $ r_j = r_j^{(1)} k^{\nu_j - \lambda} + r_j^{(2)}$ for large enough $j$. Then we find by Proposition~\ref{pr_evaluate_m} and Equation~\eqref{eq_q_1} that~\eqref{eq_sn_2} equals
\begin{align*}
	\limsup_{j\to \infty} \sum_{m' < r_j^{(1)}} &\frac{1}{g(k^{\nu_j})} \sum_{\ell = g(m_1 k^{\nu_j} + m' k^{\nu_j - \lambda})}^{g(m_1 k^{\nu_j} + (m'+1) k^{\nu_j - \lambda})-1} \inda + \frac{1}{g(k^{\nu_j})} \sum_{\ell = g(m_1 k^{\nu_j} + r_j^{(1)} k^{\nu_j - \lambda})}^{g(m_1 k^{\nu_j} + r_j)-1} \inda\\
	&= \frac{1}{k^{\lambda}} \sum_{m' < r_j^{(1)}} \sum_{1 \leq i \leq s} d_{i, \delta(q_0, (m_1k^{\lambda} + m')_k)} d(b_i(n_{\ell}), \alpha) + O\rb{\frac{1}{k^{\lambda}}}\\
	&= \frac{1}{k^{\lambda}} \sum_{m' < r_j^{(1)}} \sum_{1 \leq i \leq s} d_{i, \delta(q_1, (m')_k^{\lambda})} d(b_i(n_{\ell}), \alpha) + O\rb{\frac{1}{k^{\lambda}}}.
\end{align*}    
An analogous computation shows that~\eqref{eq_sn_1} equals the exact same expression, that is, they differ by $O(k^{-\lambda})$. As $\lambda$ was arbitrary, this shows that ~\eqref{eq_sn_1} equals ~\eqref{eq_sn_2}. Moreover, it is clear that ~\eqref{eq_shortened} is bounded form above by $\bar{d}(a(n_{\ell}), \alpha)$. Thus, we can finally apply Lemma~\ref{le_upper_density} which shows, by analogous computations for $m^{(2)} = m_1 k^{2 \lambda_2+ \lambda_3} + m_2(k^{\lambda_2} + 1) k^{\lambda_3} + m_1$, that 
\begin{align*}
	\limsup_{j\to\infty} \frac{1}{g(m^{(2)}k^{\nu_j} + r_j)} \sum_{\ell \leq g(m^{(2)}k^{\nu_j}+ r_j)} \inda \geq \bar{d}(a(n_{\ell}), \alpha).
\end{align*}
However, it is clear that it can not be strictly larger than $\bar{d}(a(n_{\ell}), \alpha)$. Thus, $m^{(2)} \in \mathcal{S}$.
A simple induction, just as in~\cite{Bell2020}, can be used to show that $m^{(h)} \in \mathcal{S}$ for any $h \in \N$.
Moreover, we find directly by Proposition~\ref{pr_evaluate_m} and Equation~\eqref{eq_q_1} that
\begin{align*}
	\limsup_{j\to \infty} \frac{1}{g(k^{\nu_j})} \sum_{h \leq g(m^{(h)}k^{\nu_j} + r_j)} \inda 
		&= k^{h\lambda_2} \underbrace{\sum_{m_1' < m_1} \sum_{1\leq i \leq s} d_{i, \delta(q_0, (m_1')_k)} d(b_i(n_{\ell}), \alpha)}_{=:\kappa_1' \in \Q}\\
			&\qquad +  \sum_{h' < h} k^{h'\lambda_2} \underbrace{\sum_{m_2' < m_2} \sum_{1 \leq i \leq s} d_{i, \delta(q_1, (m_2')_k^{\lambda_2})} 
			d(b_i(n_{\ell}), \alpha)}_{=:\kappa_2'\in \Q} + O(1).
\end{align*}
In particular
\begin{align*}
	\bar{d}(a(n_{\ell}), \alpha) &= \frac{ \kappa_1' k^{h \lambda_2} + \kappa_2' \frac{k^{h \lambda_2}-1}{k^{\lambda_2}-1} + 
	O(1)}{m_1 k^{h \lambda_2} + m_2 \frac{k^{h \lambda_2}-1}{k^{\lambda_2}-1} + O(1)}\\
		&= \frac{\kappa_1' + \frac{\kappa_2'}{k^{\lambda_2}-1} + O(k^{-h\lambda_2})}{m_1 + \frac{m_2}{k^{\lambda_2} -1} + O(k^{-h\lambda_2})}.
\end{align*}

As $h\in \N$ was arbitrary, we find in total
\begin{align*}
	\bar{d}(a(n_{\ell}), \alpha) = \frac{\kappa_1' + \frac{\kappa_2'}{k^{\lambda_2}-1}}{m_1 + \frac{m_2}{k^{\lambda_2} -1}} \in \Q,
\end{align*}
which finishes the proof for the upper density.
For the lower density we only have to use the different output function given by $\tau'(\beta) = 1$ for $\beta \neq \alpha$ and $\tau'(\alpha) = 0$.
Then the upper density of the corresponding sequence is given by $1-\underline{d}(a(n_{\ell}), \alpha)$, which concludes the proof.
\end{proof}

\bibliographystyle{amsplain}

\end{document}